\documentclass[12pt]{amsart}
\usepackage{latexsym,amsmath,amssymb,amsfonts,amscd,graphics}
\usepackage[alphabetic,lite,nobysame]{amsrefs}
\usepackage{blindtext}
\usepackage[left=1in,right=1in,bottom=1.5in]{geometry}
\usepackage[skip=4pt plus1pt, indent=10pt]{parskip}
\usepackage[mathscr]{eucal}
\usepackage{mathrsfs}
\usepackage{tikz-cd}
\usepackage{hyperref, cleveref}
\usepackage{thmtools}
\usepackage{enumitem}
\usepackage[all]{xy}
\usepackage{caption}
\usepackage{tabularx} 
\usepackage{array}    
\setlist[itemize]{leftmargin=48pt}
\hypersetup{
	colorlinks=true,
	linkcolor=magenta,
	filecolor=blue,      
	urlcolor=cyan,
	citecolor=magenta}

\numberwithin{equation}{section}
\theoremstyle{plain}
\newtheorem{theorem}[equation]{Theorem}

\newtheorem{lemma}[equation]{Lemma}
\newtheorem{proposition}[equation]{Proposition}
\newtheorem{corollary}[equation]{Corollary}

\theoremstyle{remark}
\newtheorem{remark}[equation]{Remark}

\theoremstyle{definition}
\newtheorem{definition}[equation]{Definition}

\newtheorem{conjecture}[equation]{Conjecture}

\newcommand{\bP}{\mathbb{P}}
\newcommand{\bA}{\mathbb{A}}

\newcommand{\bR}{\mathbb{R}}

\newcommand{\bZ}{\mathbb{Z}}
\newcommand{\bF}{\mathbb{F}}
\newcommand{\bC}{\mathbb{C}}

\newcommand{\bL}{\mathbb{L}}

\newcommand{\Bl}{\mathrm{Bl}}

\newcommand{\calC}{\mathcal{C}}

\newcommand{\calF}{\mathcal{F}}
\newcommand{\calK}{\mathcal{K}}

\newcommand{\calO}{\mathcal{O}}
\newcommand{\calL}{\mathcal{L}}
\newcommand{\calP}{\mathcal{P}}
\newcommand{\calI}{\mathcal{I}}

\newcommand{\Ext}{\mathrm{Ext}}
\newcommand{\Aut}{\mathrm{Aut}}

\newcommand{\Hilb}{\mathrm{Hilb}}

\newcommand{\GL}{\mathrm{GL}}
\newcommand{\Hom}{\mathrm{Hom}}
\newcommand{\Mov}{\mathrm{Mov}}
\newcommand{\Spec}{\mathrm{Spec}}

\newcommand{\coker}{\mathrm{Coker}}
\newcommand{\Gr}{\mathrm{Gr}}

\newcommand{\rank}{\mathrm{rank}}

\newcommand{\Pic}{\mathrm{Pic}}
\newcommand{\Num}{\mathrm{Num}}
\newcommand{\Nef}{\mathrm{Nef}}

\newcommand{\NS}{\mathrm{NS}}

\newcommand{\Amp}{\mathrm{Amp}}

\newcommand{\git}{/\kern-0.2em/}

\newcommand{\id}{\mathrm{id}}

\newcommand{\Ku}{\mathrm{Ku}}

\title[]{Birational cubic fourfolds via an Enriques Cremona transformation} 
\author{Corey Brooke}
\address{Department of Mathematics, Statistics, and Computer Science, St. Olaf College, Northfield, MN}
\email{brooke@stolaf.edu}

\author{Lisa Marquand}
\address{Department of Mathematics, Rutgers University - New Brunswick,
Hill Center,
110 Frelinghuysen Road,
Piscataway,
NJ}
\email{lisa.marquand@rutgers.edu}

\begin{document}

\begin{abstract}
    We describe a Cremona transformation on $\bP^5$ given by the linear system of quintics that are double along the Fano model of an Enriques surface $S$. Using this Cremona transform, we construct a birational equivalence between the very general cubic fourfold $X$ of discriminant $44$ and its unique nontrivial Fourier--Mukai partner. We also classify birational maps between smooth, non-isomorphic cubic fourfolds resolved by blowing up a smooth surface on each side: the centres must be either Fano models of Enriques surfaces, as in our construction, or Veronese surfaces, as studied by Fan--Lai.
\end{abstract}

\maketitle

\section{Introduction}

The birational geometry of smooth, complex cubic fourfolds has attracted significant recent attention over the past few decades. Most notably, the authors of \cite{KKPY25} verified the irrationality of the very general cubic fourfold, using the new theory of Hodge atoms. 
However, it is still unclear how to distinguish the birational class of one irrational cubic fourfold from another. 

There are very few known examples of  nontrivial birational maps between non-isomorphic pairs of irrational cubic fourfolds. 
To gain more traction on the geometry, one can restrict to special cubic fourfolds contained in the Hassett divisors $\calC_d$ introduced in \cite{Hassettcubicslong}, and in particular to those that are conjecturally irrational, i.e. those without an associated twisted K3 surface.
The lowest-discriminant Hassett divisors parametrising such cubics are $\calC_d$ for $d= 12$, $20$, $30$, $36$, and the subject of this article, $d=44$. 

A very general cubic fourfold in $\calC_{44}$ contains a unique smooth Enriques surface, embedded in $\bP^5$ via a Fano polarisation \cite{nuer}. 
Our main result is the following:

\begin{theorem}[Theorem \ref{thm: cubic theorem}]\label{theorem: main}
    Suppose $X$ is a very general cubic fourfold in $\calC_{44}$, containing a Fano model of an Enriques surface $S$. Then there exists a non-isomorphic cubic fourfold $Y$, also belonging to $\calC_{44}$, containing a Fano model of an Enriques surface $T$ such that $\Bl_SX\cong\Bl_TY$. In particular, $X$ and $Y$ are birationally equivalent.
\end{theorem}

For a precise characterisation of what properties are entailed in our generality assumptions, see Remark~\ref{rmk: very general}. By the specialization theorem of Kontsevich--Tschinkel \cite[Theorem~1]{KT19}, every simultaneous smooth specialization of the pairs $(X,Y)$ constructed above is again a birational pair; see Remark~\ref{rem:specialization}.

Our construction involves using the five-dimensional projective linear system of quintics in $\bP^5$ containing the Enriques surface $S$ with multiplicity $2$, restricted to $X$. After blowing up $X$ along the surface $S$, this system becomes basepoint-free, providing a morphism to $\bP^5$ whose image is $Y$. The variety $Z:=\Bl_SX\cong \Bl_TY$ may be of independent interest; indeed, it is a Fano fourfold of index 1, Picard rank $2$ and anti-canonical degree 21 (see Proposition \ref{prop: Z Fano}) . Recently, Tufo proved that the Fano dimension of a general Enriques surface is 4, by constructing a Fano fourfold $W$ with $D^b(S)\hookrightarrow D^b(W)$ as a semi-othogonal component \cite{Tufo26}. From our construction, the Fano variety $Z$ is another example of such a Fano host, and is distinct from that of Tufo.

By considering the same linear system of quintics above but on the ambient $\bP^5$ yields an interesting Cremona transform of $\bP^5$. Our next result is the following:

\begin{theorem}[Theorem \ref{thm: Cremona}]
   Let $S$ be a very general unnodal Enriques surface embedded in $\bP^5$ via a Fano polarisation. Then the linear system of quintics containing $S$ with multiplicity $2$ defines a Cremona transform $f:\bP^5\dashrightarrow \bP^5$. Further, $f$ admits a factorisation:
   \[\begin{tikzcd}
       \Bl_S\bP^5 \arrow[r, "\varphi", dashrightarrow]\arrow[d]& \Bl_T\bP^5\arrow[d]\\
       \bP^5\arrow[r, "f" above, dashrightarrow]& \bP^5
   \end{tikzcd}
   \] where $T$ is the Enriques surface constructed in Theorem~\ref{theorem: main} and $\varphi$ is an ordinary flop of 20 disjoint planes.
\end{theorem}
The base locus of the linear system on $\Bl_S\bP^5$ consists of the strict transforms of the twenty planes swept out by trisecant lines to $S$; this is exactly the flopping locus of $\varphi$. We call $f$ an Enriques Cremona transformation.

A very general Enriques surface embedded in $\bP^5$ is contained in a 9-dimensional family of cubic fourfolds. It follows a posteriori that the construction of Theorem \ref{theorem: main} holds for the very general cubic in this family, and the corresponding partners all contain the same Enriques surface $T$.

\begin{remark}
    We thus obtain a rational involution on the moduli space of Fano-polarized Enriques surfaces, interchanging $S$ and $T$. We have not verified whether this involution is trivial, although we make some preliminary comparisons between the classes of $S$ and $T$ (as well as those of $X$ and $Y$) in the Grothendieck ring of varieties in Corollary~\ref{cor: L equivalence}.
\end{remark}

The authors were inspired to consider the case of cubics in $\calC_{44}$ due to the fact that a general cubic fourfold $X\in \calC_{44}$ has a unique nontrivial Fourier--Mukai partner \cite{FL23}. Recall that the Kuznetsov component of a cubic fourfold $X$ is the full triangulated subcategory \[
\Ku(X)=\langle\calO_X,\calO_X(1),\calO_X(2)\rangle^\perp\subset D^b(X).
\] Two cubic fourfolds $X$ and $Y$ are called {\bf{Fourier--Mukai partners}} if there exists an equivalence $\Ku(X)\simeq \Ku(Y)$.  
The very general cubic fourfold has no Fourier--Mukai partners, and the number of Fourier--Mukai partners for a very general member of $\calC_d$ is counted in \cite{FL23}. These counts have been generalised to the case of more algebraic cubic fourfolds in \cite{BBM25}. Further, Huybrechts posed the following conjecture:

\begin{conjecture}\cite{huycubicbook}\label{conj: Huy}
    Suppose $X$ and $Y$ are two cubic fourfolds with $\Ku(X)\simeq \Ku(Y)$. Then $X$ and $Y$ are birational.
\end{conjecture}

Fan and Lai produced corroboratory evidence for Huybrechts's conjecture in \cite{FL23} by studying Fourier--Mukai partners for cubics in $\calC_{20}$, the first value of $d$ for which nontrivial Fourier--Mukai partnerships occur. They produce the birational Fourier--Mukai partner via a Cremona transformation whose centre is a Veronese surface---the only possible Cremona transform of $\bP^5$ with irreducible base locus \cite{CK}. 
The authors of this paper, along with Frei, produced two more families of birational Fourier--Mukai partners in \cite{BFM}, both of codimension $1$ in $\calC_{12}$. One such example, namely non-syzygetic cubic fourfold, was recast using Gale duality in \cite{BBM25}. Here, we verify the conjecture for very general cubics in $\calC_{44}$:

\begin{theorem}
    Let $X$ be a very general cubic in $\calC_{44}$ and $Y$ the cubic constructed via Theorem~\ref{theorem: main}. Then $\Ku(X)\simeq \Ku(Y)$. In particular, Conjecture \ref{conj: Huy} holds for the very general member of $\calC_{44}$.
\end{theorem}

A natural question is whether the type of construction used to verify Conjecture~\ref{conj: Huy} for Fourier--Mukai partners in $\calC_{20}$ and $\calC_{44}$ can be implemented for Fourier--Mukai partners in other discriminants. We show that such constructions are restrictive: if a birational map between smooth, non-isomorphic cubic fourfolds admits a common resolution obtained by blowing up a smooth surface on each side, then the centres are either Veronese surfaces of Fano models of Enriques surfaces. 

\begin{theorem}[Theorem \ref{thm:common-blowup}]
    Suppose $X$ and $Y$ are smooth, non-isomorphic cubic fourfolds with a birational equivalence $\varphi:X\dashrightarrow Y$.
    Suppose also that there is a resolution
    \begin{center}
        \begin{tikzcd}
            & Z \arrow[dl,"p" above] \arrow[dr,"q"] & \\
            X \arrow[rr,dashed,"\phi"] & & Y
        \end{tikzcd}
    \end{center}
    where $p$ and $q$ are blow-ups along smooth, irreducible surfaces. Then $\varphi$ is either the Cremona transform defined by a Veronese surface or an Enriques Cremona transform.
\end{theorem}

So, for example, identifying the two nontrivial Fourier--Mukai partners of a general cubic $X\in\calC_{36}$ requires new strategies; this is work in progress of the authors.

In all known examples of pairs of cubic fourfolds $X$ and $Y$ satisfying Conjecture \ref{conj: Huy}, the Fano varieties of lines on the two cubics are birational (see \cite{BFM} for detailed discussion, and \cite{BFMQ24},\cite{BBM25},\cite{BGvBM25} for more examples). The authors conjecture that if $F(X)$ and $F(Y)$ are birational, then the underlying cubic fourfolds are also birational. 
This conjecture turns out to be closely related to Conjecture~\ref{conj: Huy}: we prove in Section \ref{sec: Fano} that any pair of very general cubics $X,Y\in\calC_d$ with $F(X)$ birational to $F(Y)$ must be Fourier--Mukai partners. Our construction adds to the body of evidence supporting this second conjecture:

\begin{theorem}[Theorem \ref{thm: bir Fano}]
    Let $X$ be a very general cubic in $\calC_{44}$ and $Y$ the partner cubic constructed in Theorem \ref{theorem: main}. Then $F(X)$ is isomorphic to $F(Y)$.
\end{theorem}

\subsection*{Outline}
In Section~\ref{sec: Enriques}, we recall properties of an Enriques surface $S$ embedded in $\bP^5$ with its Fano polarisations and perform calculations related to the linear system of quintics containing $S$ with multiplicity $2$. We then restrict this linear system to a very general cubic fourfold containing $S$ in Section~\ref{sec: main const}, proving Theorem~\ref{thm: cubic theorem}. Section~\ref{sec:common-blowup} includes the proof of Theorem~\ref{thm:common-blowup}, ruling out the existence of similar constructions in other discriminants. In Section~\ref{sec: cremona}, we study the Enriques Cremona transform, proving Theorem~\ref{thm: Cremona}. We end with some lattice-theoretic computations that relate the Fano varieties of lines of the cubics related by this Cremona transform, proving Theorem~\ref{thm: bir Fano}. We provide details in Appendix \ref{appendix} on the computations undertaken in the accompanying \verb|MAGMA| code.

\subsection*{AI disclosure}
 The main construction of this paper was discovered by the authors after OpenAI's GPT-5.6 Sol suggested looking at the linear system $5H-2E$ on $\Bl_S(X)$and computing the top intersection number. We also used GPT-5.6 Sol to help write the accompanying \verb|MAGMA| code and for literature searches.  All content in the article was written by the authors, and they take full responsibility for the correctness of all results and arguments in the paper.

\subsection*{Acknowledgements} The authors would like to thank Howard Nuer for his comments and pointing out the proof of Proposition \ref{prop: unique S}. The authors also are grateful to Lev Borisov and Brendan Hassett for their interest and comments on prior versions.

\section{The Fano model of an Enriques surface}\label{sec: Enriques}

Throughout this article, $S$ denotes an unnodal Enriques surfaces, meaning an Enriques surface not containing any $-2$ curves. Every Enriques surface admits a so-called Fano polarisation:
\begin{definition}
    Let $A$ be a very ample line bundle on $S$ with $A^2=10$ and $A\cdot C\geq 3$ for every effective divisor $C$ with $C^2=0$. Then $A$ is called a \textbf{Fano polarisation} of $S$.
\end{definition}

Such a polarisation embeds $S$ into $\bP^5$ as a surface of degree $10$. We fix such an embedding $\varphi:S\hookrightarrow \bP^5$ and recall some known facts. First, the homogeneous ideal of $S$ in $\bP^5$ is generated by ten cubics, i.e. $h^0(\calI_{S}(3))=10$, and the ideal sheaf $\calI_{S}$ has the following pure resolution:
\begin{equation}\label{eqn: res of I}
    0\to \calO_{\bP^5}(-5)^{\oplus 6}\to \calO_{\bP^5}(-4)^{\oplus 15}\to \calO_{\bP^5}(-3)^{\oplus 10}\to \calI_{S}\to 0;
\end{equation}
for details, see \cite[Table 1]{SchTru}.

Moreover, there are classes $f_1,\dots,f_{10}\in\Num(S)$, called an isotropic $10$-sequence, such that $f_i^2=0$ and $3[A]=f_1+\dots+f_{10}$. Each class $f_i$ is represented by a unique curve $F_i$, embedded in $\bP^5$ as a cubic curve spanning a plane $\Pi_i$. Similarly, the unique curve $F_{-i}$ representing each divisor $f_i+K_S$ is also a cubic curve spanning a plane $\Pi_{-i}$. The trisecant variety of $S$, i.e. the locus swept out by trisecant lines to $S$ in $\bP^5$, is $\bigcup_{i=1}^{10}\Pi_i\cup\Pi_{-i}$. We record the incidence relations between these twenty planes:

\begin{lemma}\label{lem: int of planes}\cite[Section 3.5]{Enriques1}
Suppose that $S\subset\bP^5$ is a smooth Enriques surface embedded by a Fano polarisation, and let $\Pi_i$ and $\Pi_{-i}$ for $1\le i\le10$ be the planes swept out by trisecant lines to $S$, as described above. Then
\begin{enumerate}
    \item For each $i$, $\Pi_i\cap \Pi_{-i}=\varnothing$, and
    \item if $i\neq\pm j$, then $\Pi_i\cap\Pi_{j}=p_{ij}$ for smooth points $p_{ij}\in S$.
\end{enumerate}
\end{lemma}

We make use of the following observation about the normal bundle to $S$ in $\bP^5$:

\begin{lemma}\label{lem: chern classes}
    The Chern classes of the normal bundle $N_{S/\bP^5}$ are as follows:
    $$c_1( N_{S/\bP^5})= 6A+K_S, \qquad \deg c_2( N_{S/\bP^5})= 138. $$
\end{lemma}
\begin{proof}
    Restricting the Euler sequence on $\bP^5$ to $S$ yields the exact sequence
    \[
    0\to \calO_S\to \calO_S(A)^{\oplus 6}\to T_{\bP^5}|_S\to 0,
    \]
    from which we compute $c_t(T_{\bP^5}|_S)=(1+At)^6=1+6At+150t^2$, with higher-order terms vanishing since $S$ is a surface. Moreover, the short exact sequence defining the normal bundle to $S$,
    \[
    0\to T_S\to T_{\bP^5}|_S\to N_{S/\bP^5}\to 0,
    \]
    yields $c_t(T_{\bP^5}|_S)=c_t(T_S)c_t(N_{S/\bP^5})$. We have that $c_1(T_S)=-K_S$ and since $\chi(\calO_S)=1$ and $K_S^2=0$ we have $c_2(T_S)=12$ by applying Noether's formula. It follows that:
    \begin{align*}
    1+6At+150t^2&=(1-K_St+12t^2)c_t(N_{S/\bP^5})\\
    &=1+(c_1(N_{S/\bP^5})-K_S)t+(12+0+c_2(N_{S/\bP^5}))t^2,
    \end{align*}
    using the fact that $K_S$ is numerically trivial. Comparing coefficients finishes the proof.
\end{proof}

We later use the linear system of quintics in $\bP^5$ that have multiplicity 2 along $S$, coming from the sections of $\calI_S^2(5)$.

\begin{proposition}\label{prop: dim of quintics}
    For a general Fano model of an Enriques surface $S\subset\bP^5$, we have $h^0(\calI_{S}^2(5))=6$.
\end{proposition}
\begin{proof}
    First, we can twist the sequence \ref{eqn: res of I} by $\calO_{\bP^5}(5)$ to get:
    \begin{equation}
        0\to \calO_{\bP^5}^{\oplus 6}\to \calO_{\bP^5}(1)^{\oplus 15}\to \calO_{\bP^5}(2)^{\oplus 10}\to \calI_{S}(5)\to 0.
    \end{equation}
    Writing $\calK:=\coker(\calO_{\bP^5}^{\oplus 6}\to \calO_{\bP^5}(1)^{\oplus 15})$, we obtain the two short exact sequences below:
    \begin{align*}
        &0\to \calK\to \calO_{\bP^5}(2)^{\oplus 10}\to \calI(5)\to 0\\
        &0\to \calO_{\bP^5}^{\oplus 6}\to \calO_{\bP^5}(1)^{\oplus 15}\to \calK\to 0.
    \end{align*}
    The long exact sequences in cohomology yield $h^0(\calI_S(5))=126$ and  $h^i(\calI_S(5))=0$ for $i>0$.
    
    We now analyse the defining sequence for the conormal bundle, twisted by $\calO_{\bP^5}(5)$:
    \begin{equation}\label{eqn: def of conormal}
        0\to \calI_S^2(5)\to \calI_S(5)\to  N_{S/\bP^5}^\vee(5)\to 0.
    \end{equation}
    We first claim that $h^i(S,  N^\vee_{S/\bP^5}(5))=0$ for $i>0$.
    To prove this, tensor the sequence \ref{eqn: res of I} by $\calO_S(5A)$ to get the exact sequence below:
    \begin{equation} \dots \to \calO_S(A)^{\oplus 15}\to \calO(2A)^{\oplus10}\to N_{S/\bP^5}^\vee(5A)\to 0.
    \end{equation}
    Denote by $\calK:=\ker(\calO_S(2A)^{\oplus 10}\to  N_{S/\bP^5}^\vee(5A))$, and by $\calL:=\ker(\calO_S(A)^{\oplus 15}\to \calK)$. We obtain two short exact sequences: 
    \begin{align*} 
    &0\to \calK\to \calO_S(2A)^{\oplus 10}\to  N_{S/\bP^5}^\vee(5A)\to 0\\ 
    &0\to \calL\to \calO_S(A)^{\oplus 15}\to \calK\to 0. 
    \end{align*} 
    By the Kodaira vanishing theorem, $H^i(S,\calO_S(mA))=0$ for $i>0, m>0$. Since $S$ is a surface, $H^3(S,\calF)=0$ for any sheaf $\calF$. Applying the long exact sequence of cohomology to both sequences gives $H^1(S, N_{S/\bP^5}^\vee(5A))=H^2(S, N_{S/\bP^5}^\vee(5A))=0$, as desired. In particular, $h^0( N_{S/\bP^5}^\vee(5A))=\chi( N_{S/\bP^5}^\vee(5A))$.

    By Riemann--Roch for a rank-three bundle on a surface,
    $$\chi( N_{S/\bP^5}^\vee(5A))=3\chi(\calO_S)+\frac{1}{2}c_1( N_{S/\bP^5}^\vee(5A))(c_1( N_{S/\bP^5}^\vee(5A))-K_S)-c_2( N_{S/\bP^5}^\vee(5A)). $$
    Note that by Lemma \ref{lem: chern classes}, we have $c_1( N_{S/\bP^5}^\vee(5A))=c_1( N_{S/\bP^5}^\vee)+3c_1(5A) = 9A-K_S$.
    We also have:
    \begin{align*}
        c_2( N^\vee(5A))&=c_2( N_{S/\bP^5}^\vee)+2c_1( N_{S/\bP^5}^\vee)5A+3(5A)^2,\\
        &=c_2( N_{S/\bP^5}^\vee)+2(-6A-K_S)(5A)+75A^2,
    \end{align*}
    and using Lemma~\ref{lem: chern classes}, we have $\deg c_2( N_{S/\bP^5}^\vee(5A))= 138-600+750=288$. It follows that 
    \begin{align*}
        h^0( N_{S/\bP^5}^\vee(5A))=\chi( N_{S/\bP^5}^\vee(5A))&=3+\frac 12 (9A-K_S)(9A)-288\\
        &=3+405-288= 120.
    \end{align*}
    Returning to the sequence~\ref{eqn: def of conormal} and taking the long exact sequence on cohomology yields 
    \begin{align*}
        h^0(\bP^5, \calI_S^2(5))&=h^0(\bP^5,\calI_S(5))-h^0(S,  N_{S/\bP^5}^\vee(5A))+h^1(\bP^5,\calI^2_S(5))\\
        &= 126- 120+h^1(\bP^5,\calI^2_S(5))\\
        &= 6+h^1(\bP^5,\calI^2_S(5)).
    \end{align*}
    Thus the expected dimension of $H^0(\bP^5,\calI_S^2(5))$ is $6$. To show it is equal to $6$, it suffices to construct a single example. We compute this using the accompanying \verb|MAGMA| code, explained in Appendix~\ref{app: quintics} and \ref{app: lift}.
\end{proof}

Let $\pi:\bP:=\Bl_S\bP^5\rightarrow \bP^5$ be the blow up along $S\subset \bP^5$, let $H$ be the pullback of the hyperplane class, and let $E$ be the exceptional divisor. By the definition of the blow up, a quintic $q\in H^0(\bP^5,\calI^2_S(5))$ pulls back to a  section of the line bundle $\pi^*\calO_{\bP^5}(5)\otimes \calI^2_S\cdot \calO_\bP = \calO_\bP(5H-2E)$.

Consider the rational map $\bP^5\dashrightarrow \bP(H^0(\calI^2_S(5))^\vee)\cong \bP^5;$ by the above this lifts to a rational map $\bP\dashrightarrow \bP^5 $ given by the linear system $|5H-2E|$.

\begin{proposition}\label{prop: base locus}
    Let $S$ be an unnodal Fano model of an Enriques surface and $\bP:=\Bl_S\bP^5$. For general $S$, the linear system $L=5H-2E$ on $\bP$ has
    \[
    \mathrm{Base}|L|=\bigcup_{i=1}^{10}\widetilde\Pi_i\cup \widetilde\Pi_{-i}
    \]
    where $\widetilde\Pi_i$ and $\widetilde\Pi_{-i}$ are the strict transforms of $\Pi_i$ and $\Pi_{-i}$ in $\bP$, respectively. Moreover, the union above is disjoint.
\end{proposition}
\begin{proof}
    First, note that if $\ell\subset\bP^5$ is trisecant to $S$, Bezout's theorem tells us that $\ell$ is contained in each quintic in the linear system $|L|$. Recalling that the planes $\Pi_i$ and $\Pi_{-i}$ are swept out by trisecant lines to $S$, we get $\Pi_i\cup\Pi_{-i}\subset \mathrm{Base}|L|$ for each $i$.

    To verify that the base locus contains no other components, it remains to check that
    \[
    W=\bigcap_{Q\in|L|}Q
    \]
    has degree $20$ and no embedded lower-dimensional components. For this, we need only check one example, done with the accompanying \verb|MAGMA| code and explained in Appendix \ref{app: base locus} and \ref{app: lift}.

    Note that for $i,j\in\{\pm1,\dots,\pm10\}$ with $i\neq -j$, the planes $\Pi_i$ and $\Pi_j$ meet at a single point $p_{ij}\in S$ (Lemma \ref{lem: int of planes}), which coincides with the intersection point of the curves $F_i$ and $F_j$. Since the curves intersect transversely, it follows that $\widetilde{\Pi}_i$ and $\widetilde{\Pi}_j$ are disjoint. All other pairs of planes are disjoint by Lemma \ref{lem: int of planes}.
\end{proof}

\section{Cubics in $\calC_{44}$}\label{sec: main const}

In this section, we prove our main result:

\begin{theorem}\label{thm: cubic theorem}
    Let $X\in \calC_{44}$ be a very general cubic fourfold containing a (unique) very general Fano model of an Enriques surface $S$. Let $p:Z:=\Bl_SX\rightarrow X$ be the blowup, $E$ its exceptional divisor, and $H$ the pullback of the hyperplane class. Then the linear system $|5H-2E|$ defines a morphism $q:Z\rightarrow \bP^5$ that is birational onto its image. Further,
    \begin{enumerate}
        \item the image of $q$ is a smooth cubic fourfold $Y$ birational but not isomorphic to $X$,
        \item $Z$ is isomorphic to the blow up of $Y$ in a Fano embedded Enriques surface $T\subset Y$,
        \item the inverse to $X\dashrightarrow Y$ is given by the linear system $|5H'-2E'|$ on $Z$, where $H'=5H-2E$ and $E'=12H-5E$, and
        \item $Y$ is the unique nontrivial Fourier--Mukai partner of $X$; i.e., $\Ku(Y)\simeq \Ku(X)$.
    \end{enumerate}
\end{theorem}
In Section \ref{subsec: c44}, we explicitly define a \emph{very general cubic fourfold in} $\mathcal{C}_{44}$ and recall some basic properties. 
In Section \ref{subsec: int theory}, we develop the intersection theory needed to prove Theorem \ref{thm: cubic theorem}. In Section \ref{subsec: construction}, for a very general cubic $X\in \calC_{44}$ we construct the partner cubic $Y$ as well as the birational link $X\dashrightarrow Y$. Finally, in Section \ref{subsec: verification}, we verify that $X$ and $Y$ are not isomorphic and prove that $Y$ is indeed the unique nontrivial Fourier--Mukai partner of $X$.

The variety $Z$ may be of independent interest - a corollary of our result is that $Z$ is a Fano fourfold of Picard rank 2.

\subsection{Cubic fourfolds containing Fano models of Enriques surfaces}\label{subsec: c44}
Recall that the homogeneous ideal of a Fano model of an Enriques surface $S$ in $\bP^5$ is generated by ten cubics.
So, $S$ embeds in many cubic fourfolds. 
Conversely, the very general cubic fourfold $X$ in the discriminant $44$ locus $\calC_{44}$ contains a Fano model of an Enriques surface $S$ by \cite[Theorem 1.1]{nuer}. In general, $h^0( N_{S/X})=0$, so $S$ does not deform in $X$. Again under a generality assumption, the rank of $A(X):=H^{2,2}(X)\cap H^4(X,\bZ)$ is $2$, with the intersection pairing given by
\begin{center}
\begin{tabular}{r|rr}
& $\eta_X$ & $S$  \\ \hline
$\eta_X$ & $3$   & $10$   \\
$S$ & $10$   & $48$  \\
\end{tabular}
\end{center}
where $\eta_X$ is the square of the hyperplane class. It turns out that when $\rank(A(X))=2$, there is in fact only one Fano model of an Enriques surface contained in $X$, which we will prove later in Proposition~\ref{prop: unique S}.

We use the following observation in applying generality assumptions throughout the paper:

\begin{lemma}\label{lem: general X contains general S}
    A very general cubic fourfold $X$ in $\calC_{44}$ contains a Fano model of an Enriques surface $S$ such that
    \begin{itemize}
        \item $S$ contains no $-2$ curves,
        \item $h^0(\calI_{S}^2(5))=6$, and
        \item the base locus of the linear system $|5H-2E|$ on $\Bl_S\bP^5$ from Proposition~\ref{prop: base locus} consists of twenty planes and no other components.
    \end{itemize} 
\end{lemma}
\begin{proof}
Let $\mathbb{S}$ be the Hilbert scheme of Fano-embedded Enriques surfaces in $\bP^5$, which is a smooth variety of dimension $45$ by \cite[Proposition 5.7.7.]{Enriques1}, and consider the flag Hilbert scheme $I\subset\mathbb{S}\times\bP(H^0(\bP^5,\calO_{\bP^5}(3)))$ consisting of pairs $(S,X)$ such that $S\subset X$. Note that the first projection $\pi_1:I\to\mathbb{S}$ is a $\bP^9$ bundle because the homogeneous ideal of each Enriques surface $S$ is generated by ten cubics. Hence $\dim(I)=54$. Since a very general $S\in\mathbb{S}$ satisfies the three desired conditions by Propositions~\ref{prop: dim of quintics} and~\ref{prop: base locus}, the same can be said for the first component of a very general point $(S,X)\in I$. We next note that the image of the second projection $\pi_2:I\to\bP(H^0(\bP^5,\calO_{\bP^5}(3)))$ is the locus of cubic fourfolds of discriminant $44$ by \cite[Theorem 3.2]{nuer}; in particular, the image is a divisor, and the general fibre of $\pi_2$ is finite. The result follows.
\end{proof}

\begin{remark}\label{rmk: very general}
From here onward, we use the term \emph{very general cubic $X\in \calC_{44}$} to mean that $A(X)$ has rank $2$, the group of Hodge isometries of the transcendental cohomology $T(X):=A(X)^\perp\subset H^4(X,\bZ)$ consists of only $\pm \id_{T(X)}$, and $X$ contains a Fano model of an Enriques surface $S$ satisfying the three generality conditions of Lemma~\ref{lem: general X contains general S}. Recall that the lack of $-2$ curves also guarantees that  the planes $\Pi_i$ spanned by cubic curves on $S$ intersect as in Lemma~\ref{lem: int of planes}. All of these conditions hold outside a countable union of closed loci in $\calC_{44}$. In what follows, we always assume that $X$ is very general in $\calC_{44}$.
\end{remark}

For a very general cubic in $\calC_{44}$, we record the discriminant group.
\begin{lemma}\label{lem: discriminant}
    Let $X\in \calC_{44}$ be very general. Then we have 
    $D(A(X))=\bZ/44\bZ$ with generator $\left[\frac{3S-10\eta_X}{44}\right]$.
\end{lemma}
\begin{proof}
    Since $3S-10\eta_X$ is a primitive class of divisibility $44$, $\left[\frac{3S-10\eta_X}{44}\right]$ is indeed an element of $D(A(X))$ of order $44$. Moreover, the discriminant of $A(X)$ is $44$, so $|D(A(X))|=44$. The result follows.
\end{proof}

\subsection{Intersection theory}\label{subsec: int theory}

\begin{lemma}\label{lem: chern on X}
    We have $c_1(N_{S/X})= 3A+K_S$ and $c_2(N_{S/X})= 48$.
\end{lemma}
\begin{proof}
    Since $X\subset \bP^5$ is a smooth cubic hypersurface, $N_{X/\bP^5}|_S=\calO_X(3)\otimes \calO_S=\calO_S(3A)$, and $c_t(N_{X/\bP^5}|_S)=1+3At$. 
    From the exact sequence of normal bundles
    \[
    0\to N_{S/X}\to N_{S/\bP^5}\to N_{X/\bP^5}|_S\to 0,
    \]
    and Lemma~\ref{lem: chern classes}, we have 
    \[
    1+(6A+K_S)t+138t^2=c_t(N_{S/X})(1+3At).
    \]
    Comparing linear terms yields $c_1(N_{S/X})= 3A+K_S$, and then comparing quadratic terms yields $c_2(N_{S/X})=138-(3A+K_S)\cdot3A=48$.
\end{proof}

Although we will not use the following fact, it is a nice observation that the Enriques surface $S$ in $X$ is unique:

\begin{proposition}\label{prop: unique S}
    A very general cubic fourfold of discriminant $44$ contains exactly one Fano model of an Enriques surface.
\end{proposition}
\begin{proof}
    We may assume that $X$ contains at least one Fano model of an Enriques surface $S$ and that $\rank(A(X))=2$. 
    By \cite[Proposition 6.3]{nuer}, $\calI_{S/X}(2)\in\Ku(X)$, and as the author mentions, $\calI_{S/X}(2)$ is spherical. 
    Indeed, \[
    \dim\Hom(\calI_{S/X}(2),\calI_{S/X}(2))=\dim \Hom(\calI_{S/X},\calI_{S/X})=1.
    \]
    since $\calI_{S/X}^{\vee\vee}\cong \calO_X$, and for $f:\calI_{S/X}\to \calI_{S/X}$, we have $f^{\vee\vee}\in \mathrm{End}(\calO_X)=\bC$. 
    We also see that
    \[
    \Ext^1(\calI_{S/X}(2),\calI_{S/X}(2))\cong\Ext^1(\calI_{S/X},\calI_{S/X})\cong T_{[S]}\Hilb(X)=0
    \]
    since $H^0(S,N_{S/X})=0$, and by \cite[Proposition 6.3]{nuer}, $\calI_{S/X}(2)\in\Ku(X)$ so
    \[
    \Ext^2(\calI_{S/X}(2),\calI_{S/X}(2))=\Hom(\calI_{S/X}(2),\calI_{S/X}(2))^\vee
    \]
    since the Serre functor on $\Ku(X)$ is a shift by $2$.
    
    Now, suppose for the sake of contradiction that $X$ contains another Fano model of an Enriques surface, $S'$. 
    Then $\deg(S')=10$, and $S'$ gives a discriminant $44$ marking on $X$, so $(S')^2=48$. Writing $S'=a\eta_X+bS$ in $A(X)$, we have
    \[
    3a+10b=10\qquad\text{and}\qquad 3a^2+20ab+48b^2=48.
    \]
    The only integer solution is $a=0$ and $b=1$, so $[S]=[S']$. 

    Grothendieck--Riemann--Roch gives
    \[
    c_t(\calI_{S/X})=1+[S]t^2+i_*c_1(N_{S/X})t^3+i_*(c_1(N_{S/X})^2)t^4,
    \]
    where $i:S\to X$ is the inclusion. By Lemma~\ref{lem: chern on X}, $c_1(N_{S/X})=3i^*h_X+K_S$, and $K_S$ is numerically trivial, so 
    \[
    c_t(\calI_{S/X})=1+[S]t^2+3h_X[S]t^3+90t^4. 
    \]
    Performing the same calculation for $S'$ and using that $[S]=[S']$, we find that $[\calI_{S/X}]=[\calI_{S'/X}]$ in the numerical Grothendieck group.

     In particular,
    \[
    \chi(\calI_{S/X}(2),\calI_{S'/X}(2))=\chi(\calI_{S/X}(2),\calI_{S/X}(2))=2.
    \]
    
    On the other hand, $S'\not\subset S$, so 
    \[
    \Hom(\calI_{S/X}(2),\calI_{S'/X}(2))\cong\Hom(\calI_{S/X},\calI_{S'/X})=0,
    \]
    and again using \cite[Proposition 6.3]{nuer}, we have $\calI_{S'/X}(2)\in\Ku(X)$, so
    \[
    \Ext^2(\calI_{S/X}(2),\calI_{S'/X}(2))\cong\Hom(\calI_{S'/X}(2),\calI_{S/X}(2))^\vee=0
    \]
    and
    \[
    \Ext^i(\calI_{S/X}(2),\calI_{S'/X}(2))\cong \Ext^{2-i}(\calI_{S'/X}(2),\calI_{S/X}(2))^\vee=0
    \]
    for $i>2$.
    So,
    \begin{align*}
    2=\chi(\calI_{S/X}(2),\calI_{S'/X}(2))&=-\dim\Ext^1(\calI_{S/X}(2),\calI_{S'/X}(2))\le0,
    \end{align*}
    a contradiction.
\end{proof}

Let $Z=\Bl_S X$, with $p:Z\rightarrow X$ the blow up. Denote by $H$ the pullback of the hyperplane class and by $E$ the exceptional divisor. 
\begin{lemma}\label{lem: int no. Z}
    We have $\Pic(Z)=\bZ[H]\oplus \bZ[E]$ with the following intersection numbers:
    $$H^4=3, \qquad H^3E=0, \qquad H^2E^2=-10, \qquad HE^3=-30, \qquad E^4=-42. $$
\end{lemma}
\begin{proof}
    The first two intersection numbers are clear since $S$ has codimension two in the cubic fourfold $X$.

    Recall that $E=\bP(N_{S/X})$, and we have the following diagram:
    $$\begin{tikzcd}
        E\arrow[r, hookrightarrow, "j"]\arrow[d, "\pi"]& Z\arrow[d,"p"]\\
        S\arrow[r, hookrightarrow, "i"]& X
    \end{tikzcd}$$
    Using notation from \cite[Section 13.6]{3264}, write $\zeta:=c_1(\calO_E(1))$. Then $j^*E=-\zeta$.
    We also have $j^*H=\pi^*A$, and for any class $\alpha\in A(Z)$ we have $\alpha E=j_*(j^*\alpha)$.
    
    From \cite[Theorem 9.6]{3264}, the chow ring of $E$ is
    $$A^*(E)=A^*(S)[\zeta]/(\zeta^2+\pi^*c_1(N_{S/X})\zeta+\pi^*c_2(N_{S/X})), $$
    hence $\zeta^2=-\pi^*c_1(N_{S/X})\zeta-\pi^*c_2(N_{S/X})$. Pushing this relation forward and applying the push-pull formula (\cite[Lemma 9.7]{3264}) gives $\pi_*(\zeta^2)=-c_1(N_{S/X})$, and multiplying the relation by $\zeta$ and pushing forward gives $\pi_*(\zeta^3)= c_1(N_{S/X})^2-c_2(N_{S/X})$.

    We compute $H^2E^2=j_*(j^*(H^2E))$.
    Thus we compute 
    $$j^*(H^2E)=-\pi^*(A^2)\zeta= -A^2\pi_*(E)=-A^2=-10, $$
     where the second equality is via pushing forward to $S$. 

     Next $HE^3= j_*(j^*(HE^2))$. Thus 
     $j^*(HE^2)=\pi^*(A)j^*(E^2)=\pi^*A\zeta^2$, and pushing down to $S$ gives $$A\pi_*(\zeta^2)=-A c_1(N_{S/X})=- A(3A+K_S)=-30, $$
     where the second equality is applying Lemma \ref{lem: chern on X}.

    Finally, $E^4=j_*(j^*E^3)$. Again, since $j^*E=-\zeta$, we have 
    $j^*(E^3)= -\zeta^3$. Pushing down to $S$ gives $-\pi_*(\zeta^3)=c_2(N_{S/X})-c_1(N_{S/X})^2= 48-90=-42, $ again using Lemma \ref{lem: chern on X}.
\end{proof}

Now we let $H':= 5H-2E$, and $E'=12H-5E$.
We record some observations:
\begin{lemma}\label{lem: H' and E' numerics}
    We have:
    $$(H')^4=3,\qquad 3H'-E'=3H-E=-K_Z, $$
    and $H+H'=-2K_Z$.
\end{lemma}
\begin{proof}
    Immediate from the above computation.
\end{proof}

\subsection{Constructing the partner cubic}\label{subsec: construction}

The numerics  in Lemma~\ref{lem: H' and E' numerics} suggest that $H'$ and $E'$ might be the pullback of the hyperplane class and the exceptional divisor under some other blowup $q:Z\to Y$, where $Y$ is a cubic fourfold. In this subsection, we verify that this is indeed the case, showing moreover that $Y$ contains a Fano model of an Enriques surface $T$ and that $q$ is the blowup of $Y$ along $T$.

\begin{proposition}\label{prop: image is a cubic}
    For a very general $X\in\calC_{44}$, the linear system $|H'|=|5H-2E|$ on $Z:=\Bl_SX$ is basepoint-free and induces a morphism 
    $q:Z\rightarrow \bP^5$ whose image $Y$ is a smooth cubic fourfold birational to $X$.
\end{proposition}
\begin{proof}
    First, note that the linear system $|H'|$ is the restriction of the linear system $|L|$ on $\bP^5$ from Proposition~\ref{prop: base locus}; that is, its sections are given by quintics that are double along $S$.
    
    A very general cubic fourfold $X\in \calC_{44}$ does not contain any planes.
    In particular, $X$ contains none of the planes $\Pi_i$ for $i=\pm1,\dots,\pm10$ swept out by trisecant lines to $S$.
    Hence each of those planes intersects $X$ in the smooth cubic curve $F_i=\Pi_i\cap S$.
    Since this intersection is confined to $S$, the strict transform $\widetilde{\Pi}_i\subset\bP$ is disjoint from $Z$. Then $Z$ does not intersect the base locus of $|L|$ by Proposition~\ref{prop: base locus}, so the linear system $|H'|$ on $Z$ is basepoint-free.

    Since $\dim |H'|=5$, the linear system determines a morphism $q:Z\to\bP^5$ whose image, which we call $Y$, is nondegenerate. Further, from Lemma \ref{lem: int no. Z}, we know that $(H')^4=3\neq0$, so $Y$ has dimension $4$. Moreover, $(H')^4=\deg(q)\deg(Y)$, so either $\deg(Y)=1$ or $\deg(Y)=3$; the first case cannot happen since $Y$ is nondegenerate. Hence $q:Z\rightarrow Y$ is a birational morphism onto a cubic fourfold.

    For the smoothness of $Y$, it suffices to demonstrate one example where $Y$ is smooth. We provide such an example using the accompanying \verb|MAGMA| code, explained in Appendix \ref{app: smoothness} and \ref{app: lift}.
\end{proof}
 Thus we have a diagram:
 $$\begin{tikzcd}
     &Z\ar[dl, "p" above] \ar[dr, "q"]&\\
     X\ar[rr, dashed]&& Y
 \end{tikzcd}. $$
where both $p$ and $q$ are birational maps. Next, we will argue that $q$ is also the blow up in a Fano-embedded Enriques surface.

\begin{proposition}\label{prop: Y in C44}
    For a very general $X\in\calC_{44}$, the morphism $q:Z\rightarrow Y$ is the blow up of $Y$ along a smooth Enriques surface $T\subset Y$, where $T\subset \bP^5$ is the Fano embedding. In particular, $Y\in \calC_{44}$.
\end{proposition}

The proof amounts to a sequence of technical lemmas. In Lemma~\ref{lem: q is a blowup}, we use structural results about contractions of fourfolds to prove that $q$ is the blowup of $Y$ along a smooth surface $T$, relying on Lemma~\ref{lem: no surface fibres} to rule out the existence of certain two-dimensional fibres. We then use numerics to identify the surface $T$ as an Enriques surface with its Fano embedding in Lemma~\ref{lem: T is Enriques}.

\begin{lemma}\label{lem: q is a blowup}
    The exceptional divisor $E'$ of $q:Z\to Y$ is irreducible, and $T:=q(E')$ is an irreducible surface.
\end{lemma}
\begin{proof}
    Since $Y$ is normal and $q$ is birational, the fibres of $q$ are connected. 
    By the Lefschetz hyperplane theorem, the Picard ranks of $X$ and $Y$ are $1$, and the Picard rank of $Z$ is $2$ since it is the blowup of $X$ along an irreducible surface. Hence $q$ must consist of at least one divisorial contraction, and the numerical classes of the contracted curves span a single ray in the Mori cone of $Z$. Since $Y$ is smooth, a nontrivial projective birational morphism onto $Y$ cannot be small.
    Thus the relative canonical divisor is an effective exceptional divisor: 
    \[
    K_{Z/Y}=K_Z-q^*K_Y=\sum a_i E_i
    \]
    with $a_i>0$. Note that $K_{Z/Y}$ cannot be $q$-nef; indeed this follows from applying Lemma \cite[Lemma 3.39]{KolMori} to $B:=K_{Z/Y}$. Thus if $C$ is a curve contracted by $q$, we have 
    $$0>C\cdot K_{Z/Y}=C\cdot K_Z. $$ Thus $K_Z$ is negative on the entire contraction ray. In fact, $-K_Z$ is $q$-ample, and so $q:Z\rightarrow Y$ is an elementary $K_Z$-negative contraction.
    Then there exists a prime exceptional divisor $E'$ that is negative on the ray, and any curve $C$ contracted by $q$ is contained in $E'$. In particular, $K_{Z/Y}=mE'$

    Since $H'=5H-2E$, we have
    \[
    mE'=K_{Z/Y}=(-3H+E)+3H'=12H-5E.
    \]
    Since $12H-5E$ is primitive, $m=1$,
    and using Lemma~\ref{lem: int no. Z}, we compute $(H')^2(E')^2=-10$. It follows that $T:=q(E')$ is an irreducible surface: indeed, if $q(E')$ had dimension at most 1, then $q_*(E')^2=0$, contradicting the intersection number via the projection formula. 
\end{proof}

\begin{lemma}\label{lem: no surface fibres}
    There is no irreducible rational surface $F$ contained in a fibre of $q:Z\to Y$ with $(-K_Z|_F)^2\in\{1,2\}$.
\end{lemma}
\begin{proof}
    Suppose to the contrary that $F$ is a rational surface embedding in a fibre of $q$, and $L_F^2=k\in \{1,2\}$ where $L_F=-K_Z|_F$.
    We have $L_F=-K_Z|_F= -(q^*K_Y+E')|_F$. Since $q(F)$ is a point, any line bundle pulled back from $Y$ restricts trivially to $F$, so $E'|_F=-L_F$.
    Further, since $H'=q^*\calO_Y(1)$, we see $H'|_F=\calO_F$.
    Writing $H=5H'-2E'$ and $E=12H'-5E'$, we obtain
    $$H|_F= 2L_F, \qquad E|_F=5L_F. $$
    From this, we compute $H^2F=4k$, $HEF=10k$, and $E^2F=25k$.
    In particular, $R:=p(F)$ is a surface of degree $4k$ in $X$.
    
    Consider $[F]\in H^4(Z,\bZ)$. There exists class $r\in H^4(X,\bZ)$, $\alpha\in H^2(S,\bZ)_{\mathrm{tf}}$ such that $[F]=p^*(r) +j_*p^*\alpha$, thus $r=p_*([F])$. The class $r$ is integral and algebraic, and so $r=[R]$. 
    Thus by the projection formula,
    \begin{align*}
        4k&=H^2F= \eta_X R;\\
        10k&=HEF= -A\alpha;\\
        25k&=E^2F= -[S]r-c_1(N_{S/X})\alpha.
    \end{align*}

    Since $c_1(N_{S/X})=3A+K_S$ and $K_S$ is torsion, $c_1(N_{S/X})\alpha= 3A\alpha$.
    Thus $25k=-[S]r+30k$, and so $[S]r=5k$.

    Now $r\in A(X)$, and so we can write as $$r=a\eta_X+b[S]. $$ Using $[S]r=5k$ and $\eta_Xr=4k$, we obtain:
    \begin{align*}
        3a+10b&=4k,\\
        10a+48b&=5k,
    \end{align*}
    and solving gives $a=\frac{71k}{22}, b=-\frac{25k}{44}$. For $k=1$ or 2, neither coefficient is integral, a contradiction.
\end{proof}

\begin{lemma}
    The surface $T$ is smooth, and $q:Z\to Y$ is the blowup along $T$.
\end{lemma}
\begin{proof}
    We first show that $q$ has no two-dimensional fibres. 
    Any two-dimensional fibres are isolated since $E'$ is a threefold and $T=q(E')$ is a surface. Further, by the proof of Lemma \ref{lem: q is a blowup}, $-K_Z$ is $q$-ample, and so it follows that $q:Z\rightarrow Y$ is a `good contraction' \emph{sensu}  \cite{MW}.
    If $(F,L_F)$ is a component of a two-dimensional fibre, and $L_F=-K_Z|_F$, then applying \cite[Proposition 4.11]{MW}, we have $(F,L_F)=(\bP^2, \calO_{\bP^2(1)})$, or $(Q,\calO_Q(1))$, where $Q$ is a (possibly singular) quadric with its hyperplane polarisation. Neither case is possible by Lemma~\ref{lem: no surface fibres}.

    Since $q:Z\rightarrow Y$ has no two dimensional fibres, it follows that $q$ is an extremal divisorial contraction with unique exceptional divisor $E'$. Recall that we already computed $\dim T=2$. The main result of \cite[Theorem 2.3]{Ando} implies that $T$ is smooth, and $Z\cong \Bl_TY$.
\end{proof}

\begin{lemma}\label{lem: T is Enriques}
    $T$ is an Enriques surface with its Fano embedding in $\bP^5$.
\end{lemma}
\begin{proof}
    The different blowups give us two decompositions of the torsion-free cohomology of $Z$, including the following:
    \begin{equation}\label{eqn: decomp of H^4(Z)}
    \begin{aligned}
    H^4(Z,\bZ)_{\mathrm{tf}}&\simeq H^4(X,\bZ)\oplus H^2(S,\bZ)_{\mathrm{tf}}(-1)\\
    &\simeq H^4(Y,\bZ)\oplus H^2(T,\bZ)_{\mathrm{tf}}(-1).
    \end{aligned}
    \end{equation}
    and 
    \begin{equation}\label{eqn: decomp of H^3(Z)}
    \begin{aligned}
    H^3(Z,\bZ)_{\mathrm{tf}}&\simeq H^3(X,\bZ)\oplus H^1(S,\bZ)_{\mathrm{tf}}(-1)\\
    &\simeq H^3(Y,\bZ)\oplus H^1(T,\bZ)_{\mathrm{tf}}(-1).
    \end{aligned}
    \end{equation}
    Since $X$ and $Y$ are both smooth cubic fourfolds, they have the same Hodge numbers, so the surfaces $S$ and $T$ also have the same Hodge numbers, yielding the following:
    \[
    q(T)=p_g(T)=0,\qquad \chi(\calO_T)=1,\qquad\chi_{\mathrm{top}}(T)=12.
    \]
    By Noether's formula, $K_T^2=12\chi(\calO_T)-\chi_{\mathrm{top}}(T) =0$.

    Now, let $h_Y$ be the hyperplane class on $Y$ and $A_T=h_Y|_T$. Note $H'=q^*h_Y$. The projection formula gives
    \[
    (H')^2(E')^2= (q^*h_Y)^2(E')^2= h_Y^2\cdot [-T],
    \]
    using the fact that $q_*((E')^2)=-[T]$. On the other hand, $(H')^2(E')^2=-10$ by Lemma~\ref{lem: int no. Z}, so $T$ is a surface of degree 10 in $\bP^5$.

    Note that $(E')^3=j_*j^*(E'^2)=j_*(\zeta'^2)$, where $\zeta'=c_1(\calO_{E'}(1))$.
    Similarly, 
    \[
    -30=H'(E')^3=h_Yq_*((E')^3)=h_Yq_*(j_*(\zeta'^2))= h_Yi'_*\pi'_*(\zeta'^2)= -A_Tc_1(N_{T/Y}),
    \]
    where we use that $\pi'_*(\zeta'^2)=-c_1(N_{T/Y})$, as in Lemma \ref{lem: int no. Z}.  Using adjunction \cite[Chapter II, Proposition 8.20]{hartshorne}, $K_T\cong (K_Y+\det N_{T/Y})|_T=-3A_T+c_1(N_{T/Y})$. 
    Hence intersecting with the ample class $A_T$ gives $A_T\cdot K_T=0$.
    By the Hodge index theorem, $A_T^\perp$ is negative definite since $A_T$ is ample. 
    Thus the fact that $A_T\cdot K_T=0$ and $K_T^2=0$ implies that $K_T\sim_{\mathrm{num}} 0$. Adjunction then shows that $T$ is minimal. Together with $q(T)=p_g(T)=0$, this implies that $T$ is an Enriques surface (applying \cite[Theorem VIII.2]{beauville}). 
\end{proof}

Note that $Y$ has $\rank (A(Y))=\rank(A(X))=2$ and contains the Enriques surface $T$; it follows that $Y\in\calC_{44}$.
This completes the proof of Proposition~\ref{prop: Y in C44}. 

\subsection{Verifying the Fourier--Mukai partnership}\label{subsec: verification}
We summarize the setup obtained in the previous section with the following diagram:
\[
\begin{tikzcd}
     E \ar[rr,hook,"j"] \ar[d,"\pi" left] & &Z\ar[dl, "p" above] \ar[dr, "q"]& & E' \ar[d,"\pi'"] \ar[ll,hook',"j'" above]\\
     S \ar[r,hook,"i"] & X\ar[rr, dashed]&& Y & T \ar[l,hook',"i'" above]
 \end{tikzcd}
 \]
We now show that the cubic fourfold $Y$ constructed in the previous subsection is the unique Fourier--Mukai partner of $X$. We begin by showing that $X$ and $Y$ are not isomorphic.

\begin{proposition}
    Let $X\in \calC_{44}$ be very general and $Y$ constructed as above. Then $Y$ is not isomorphic to $X$.
\end{proposition}
\begin{proof}
    Again, we look at the two decompositions of the torsion-free cohomology of $Z$:
   \begin{equation}\label{eqn:hodge decomp Z}
    \begin{aligned}
    H^4(Z,\bZ)_{\mathrm{tf}}&\simeq p^*H^4(X,\bZ)\oplus j_*\pi^*H^2(S,\bZ)_{\mathrm{tf}}(-1)\\
        &\simeq q^*H^4(Y,\bZ)\oplus j'_*\pi'^*H^2(T,\bZ)_{\mathrm{tf}}(-1).
    \end{aligned}
    \end{equation}
In particular, this is an orthogonal decomposition for the intersection pairing. Retricting to the algebraic cohomology of $Z$, we get:
\begin{align*}
    A(Z)_{\mathrm{tf}}&\simeq p^*A(X)\oplus \mathrm{Num}(S)(-1)\\
    &\simeq q^*A(Y)\oplus \mathrm{Num}(T)(-1)
\end{align*}
For an Enriques surface, $\mathrm{Num}(S)(-1)\simeq U\oplus E_8;$ in particular it is unimodular.
It follows then that $D(A(Z))\cong D(p^*A(X))\cong D(q^*A(Y))$ as discriminant groups (i.e. respecting the bilinear forms).

Note that $H^4(Z,\bZ)_{\mathrm{tf}}$ is a unimodular lattice (with the intersection pairing), and in fact is a finite index overlattice $$A(Z)\oplus T(Z)\subset H^4(Z,\bZ)_{\mathrm{tf}}. $$ 
By Nikulin's theory of primitive embeddings \cite[Proposition 1.15.1]{nikulin}, there is an anti-isometry $\gamma_Z:D(A(Z))\simeq D(T(Z))$.
Similarly, we have anti-isometries:
\begin{align*}
    \alpha_X&:D(A(X))\rightarrow D(T(X)),\\
    \alpha_Y&:D(A(Y))\rightarrow D(T(Y)).
\end{align*}
Take $\delta_X:=\frac{-10\eta_X+3S}{44}$ as the generator of $D(A(X))$ and $\delta_Y:=\frac{-10\eta_Y+3T}{44}$ as the generator of $D(A(Y))$.
Then both $p^*\delta_X$ and $q^*\delta_Y$ are generators of $D(A(Z))$. We claim that $\gamma_Z(p^*\delta_X)=p^*\alpha_X(\delta_X)$, and similarly $\gamma_Z(q^*\delta_Y)=q^*\alpha_Y(\delta_Y)$.
Indeed, the isometry $\alpha_X:D(A(X))\to D(T(X))$ is defined as follows: if $a\in A(X)^\vee$, then choose $t\in T(X)^\vee$ such that $a+t\in H^4(X,\bZ)$. Then $\alpha_X(a)=[t]\in D(T(X))$. Applying $p^*$, we have $p^*(a)+p^*(t)\in p^*H^4(X,\bZ)\subset H^4(Z,\bZ)_{\mathrm{tf}}$,
and so (again using the fact that $H^4(Z,\bZ)_{\mathrm{tf}}$ is unimodular) $\gamma(p^*(a))=[p^*(t)]=p^*[\alpha(a)]$.

Now we compare the generators $p^*\delta_X$ and $q^*\delta_Y$. By the proof of \cite[Theorem 13.14]{3264}, we have that $q^*[T]=j_*(\zeta +\pi^*c_1(N_{T/Y}))$. Since $j_*(\zeta)=-E'^2$ we have:
\begin{align*}
    q^*(\eta_Y)&=(H')^2,\\
    q^*(T)&=-E'^2+j_*\pi^*c_1(N_{T/Y})=-E'^2+j_*\pi^*(3A_T+K_T)=-E'^2+3H'E'.
\end{align*}
Using the relations $H'=5H-2E$ and $E'=12H-5E$, we see that
$$q^*(\delta_Y)=\frac{1}{44}(-142H^2+119HE-25E^2). $$

We also have $$p^*(\delta_X)=\frac{1}{44}(-10H^2-3E^2+9HE), $$ so
$$q^*(\delta_Y)=23p^*(\delta_X)+(2H^2-2HE+E^2), $$ and it follows that $q^*(\delta_Y)=23p^*(\delta_X)$ in $D(A(Z))$. Since $23^{-1}\cong 23 \mod 44$, we also have $p^*(\delta_X)=23q^*(\delta_Y)$. Set $\epsilon_X:=\alpha_X(\delta_X)$ and $\epsilon_Y:=\alpha_Y(\delta_Y)$, noting that $\epsilon_X$ and $\epsilon_Y$ are generators of $D(T(X))$ and $D(T(Y))$, respectively. Then
\begin{align*}
    p^*(\epsilon_X)&=p^*(\alpha_X(\delta_X))\\
    &=\gamma_Z(p^*(\delta_X))\\
    &=23\gamma_Z(q^*(\delta_Y))\\
    &=23 q^*(\alpha_Y(\delta_Y))\\
    &=23 q^*(\epsilon_Y).
\end{align*}

Now suppose that there is some isomorphism $f:X\rightarrow Y$. Pullback induces an isomorphism $f^*:T(Y)\rightarrow T(X)$ hence also $f^*:D(T(Y))\to D(T(X))$. 
Since the Picard rank of $X$ (and $Y$) is one, $f^*(h_Y)=h_X$, so $f^*(\eta_Y)=\eta_X$. 
Because $A(X)$ for a general cubic fourfold in $X\in\calC_{44}$ contains only one class with the numerics of a Fano model of an Enriques surface (indeed the surface $S$ is unique), we also have $f^*([T])=[S]$. 
By the computations above, $f^*(\epsilon_Y)=\epsilon_X$.

We next track the action on the generator $\gamma(q^*(\delta_Y))=q^*\alpha_Y(\delta_Y)=q^*\epsilon_Y$ of $D(T(Z))$  through the isomorphism
\[
p^*\circ f^*\circ (q^*)^{-1}:T(Z)\to T(Y)\to T(X)\to T(Z).
\]
We see that:
\[
p^*\circ f^*\circ (q^*)^{-1}(q^*\epsilon_Y)=p^*\circ f^*(\epsilon_Y)=p^*(\epsilon_X)=23q^*\epsilon_Y.
\]
Thus we have exhibited a  nontrivial Hodge isometry of $T(Z)$ acting not as $\pm \id$ on $D(T(Z))$. Since $T(Z)\cong T(X)$, this is also a Hodge isometry of $T(X)$, and thus a contradiction since $X$ is very general. 
\end{proof}

By \cite[Proposition 2.6]{FL23}, the very general member $X\in\calC_{44}$ has a unique Fourier--Mukai partner. That is, there is a unique cubic fourfold $Y$ such that $X\not\cong Y$ but $\Ku(X)\simeq \Ku(Y)$. 
For a very general cubic fourfold $X\in \calC_d$, one can detect Fourier--Mukai partners via the \textbf{Addington--Thomas} lattice $\widetilde{H}(\Ku(X),\bZ)$, \cite{AT14}. As an abstract lattice, $\widetilde{H}(\Ku(X),\bZ)\cong \widetilde{\Lambda}$, where 
$$\tilde{\Lambda}:=E_8(-1)^{\oplus 2}\oplus U^{\oplus 4}, $$ the extended K3 lattice.
We will make use of the following theorem:
\begin{theorem}\cite[Theorem 1.5 (iii)]{HuyK3cat}\label{thm: huy17}
    Let $X$ and $Y$ be smooth cubic fourfolds. For arbitrary $d$ and very general $X\in \calC_d$, there exists a Fourier--Mukai equivalence $\Ku(X)\simeq \Ku(Y)$ if and only if there exists a Hodge isometry $\widetilde{H}(\Ku(X),\bZ)\cong \widetilde{H}(\Ku(Y),\bZ)$.
\end{theorem}

We now show that this Fourier--Mukai partner of $X$ is the constructed cubic fourfold $Y$, completing the proof of Theorem~\ref{thm: cubic theorem}.

\begin{proposition}
    Let $X\in \calC_{44}$ be very general, and $Y$ the cubic constructed as above. Then $\Ku(X)\simeq \Ku(Y)$.
\end{proposition}
\begin{proof}
We again analyse the two decompositions of $H^4(Z,\bZ)_{tf}$ as in Equation \ref{eqn: decomp of H^4(Z)}.
An Enriques surface has $h^{2,0}=0$, so the entire torsion-free cohomology of $S$ (resp. $T$) is algebraic. It follows that we have Hodge isometries:
$$T(X)\simeq T(Z)\simeq T(Y)$$ induced from pullback.
By \cite[Proof of Prop 2.6]{FL23}, this implies that $X$ and $Y$ are Fourier Mukai partners. More precisely, since $44$ is not divisible by $9$, we can apply \cite[Theorem 1.14.4]{nikulin} to see that $T(X)$ admits a single embedding into the Addington--Thomas lattice $\widetilde{H}(\Ku(X),\bZ)$. Hence one can lift the isometry $T(X)\simeq T(Y)$ to one of $\tilde{H}(\Ku(X),\bZ)\simeq \tilde{H}(\Ku(Y), \bZ)$. By Theorem \ref{thm: huy17}, for a very general $X\in \calC_d$, this is equivalent to an equivalence $\Ku(X)\simeq \Ku(Y)$.
\end{proof}

\begin{remark}[Specialization]
\label{rem:specialization}
Birationality persists under simultaneous smooth specialization
of the pairs constructed above. More precisely, let
$\mathcal{X}\rightarrow B$ and $\mathcal Y\longrightarrow B$ be smooth projective families of cubic fourfolds over a smooth connected complex curve, whose geometric generic fibres are related by the construction of Theorem \ref{thm: cubic theorem}. 
Then $\mathcal X_b$ and $\mathcal Y_b$ are birational for every $b\in B(\mathbb C)$.
Indeed, the birational map between the geometric generic fibres is defined over the function field of the base. The assertion then follows from \cite[Theorem~1]{KT19}.

Neither the Enriques surfaces nor the common blow-up presentation need extend to the special fibre. Further, it is unclear to us whether the Fourier-Mukai partnership described above specializes.
\end{remark}

\subsection{The variety $Z$}
Let $Z=\Bl_SX$ for $X, S$ as in Theorem \ref{theorem: main}. In fact, we immediately obtain the following:

\begin{proposition}\label{prop: Z Fano}
    Let $X\in \calC_{44}$ containing the Enriques surface $S$, and $Z=\Bl_S X.$ Then $Z$ is a smooth Fano fourfold of Picard rank 2 and anticanonical degree $(-K_Z)^4=21.$ Further,
    $$\Nef(Z)= \bR_{\geq 0}H+\bR_{\geq 0}(5H-2E),$$ with the two elementary extremal contractions $p:Z\to X$ and $q:Z\to Y.$
\end{proposition}
\begin{proof}
    We can write $K_Z=-3H+E$, and simultaneously $K_Z=-3H'+E'.$ In particular, $-K_Z=\frac{1}{2}(H+H')$. The sum of two nef boundary classes is ample, and so $-K_Z$ is ample. The remaining claims are immediate.
\end{proof}

Note that $Z$ is a Fano fourfold of K3-type, but is simultaneously a Fano host of the Enriques surfaces $S$ and $T$, i.e. $D^b(S)\hookrightarrow D^b(Z)$ as a semi-orthogonal component, as seen using the blowup formula. 
\begin{remark}
It is possible that one could obtain a relation between $D^b(S)$ and $D^b(T)$ by mutating the two semi-orthogonal decompositions of $Z$. We did not pursue this here.    
\end{remark}

\section{Type II Sarkisov links between cubic fourfolds}\label{sec:common-blowup}

The birational links that explain Fourier--Mukai partnerships in $\calC_{20}$ and $\calC_{44}$ share a striking feature: in both cases, the Fourier--Mukai partners $X$ and $Y$ contain smooth, irreducible surfaces $S$ and $T$, respectively, with $\Bl_SX\cong\Bl_TY$. In both cases, the birational map between them is the simplest example of a Type II Sarkisov link \cite{HM13}. We now prove that such links are unique to these discriminants.

\begin{theorem}\label{thm:common-blowup}
    Suppose $X$ and $Y$ are smooth, non-isomorphic cubic fourfolds with a birational equivalence $\varphi:X\dashrightarrow Y$.
    Suppose also that there is a resolution
    \begin{center}
        \begin{tikzcd}
            & Z \arrow[dl,"p" above] \arrow[dr,"q"] & \\
            X \arrow[rr,dashed,"\phi"] & & Y
        \end{tikzcd}
    \end{center}
    where $p$ is the blow-up of $X$ along a smooth surface $S$, and $q$ is the blow-up of $Y$ along a smooth surface $T$. Let $H$ and $H'$ be the pullbacks of the hyperplane classes of $X$ and $Y$ to $Z$, respectively, and let $E$ and $E'$ be the exceptional divisors.
    Then either
    \begin{enumerate}
        \item both centres are Veronese surfaces, and $H'=2H-E$, or
        \item both centres are degree 10 Fano models of Enriques surfaces, and $H'=5H-2E.$
    \end{enumerate}
    In particular, this only occurs when $X,Y\in\calC_d$ for $d=20$ (in case 1) or $44$ (in case 2).
\end{theorem}
\begin{proof}
    Since $H$ and $E$ generate $\Pic(Z)$, we can write $H'=aH-bE$ for some integers $a$ and $b$.
    Let $f$ be a fibre of $E\to S$. Then $H\cdot f=0$ and $E\cdot f=-1$. It follows that $H'\cdot f=b\geq 0$ since $H'$ is nef; our assumption that $X\not\cong Y$ gives $b>0$.

    Since $-3H+E=K_Z=-3H'+E',$ we have $E'=3(a-1)H-(3b-1)E.$ Both $\{H, E\}$ and $\{H',E'\}$ are integral basis for $\Pic(Z),$ so the change of basis matrix 
    \[
    M:=\left(\begin{matrix}
        a & -b\\
        3(a-1) & -(3b-1)
    \end{matrix}\right)
    \]
    lies in $\GL_2(\bZ)$; taking the determinant yields $a-3b =\pm 1.$
    Suppose that $a-3b=1,$ and let $f'$ be a fibre of $E'\to T.$ Then $H\cdot f' = (-(3b-1)H'+bE')\cdot f' =-b<0,$ contradicting the nefness of $H$. 
    Hence $a-3b=-1$. 
    From this, we also obtain 
    \[
    H+H'=(a+1)H-bE=3bH-bE=-b K_Z,
    \]
    and
    \begin{equation}\label{eqn:H in terms of H',E'}
    H=-H'-bK_Z=-H'+3bH'-bE'=aH'-bE'.
    \end{equation}

    We now use some intersection theory on $Z$. Set $$e:=H^2\cdot S, \qquad k:=H|_SK_S, \qquad c:=\deg c_2(S).$$ Note that $e=\deg(S)$, $c_1(N_{S/X})=3H+K_S$, and $\deg c_2(N_{S/X})=6e+3k+K^2_S-c$. We have the following intersection numbers on $Z$:
    $$H^4=3, \qquad H^3E=0,\qquad H^2E^2=-e,  \qquad HE^3=-(3e+k),\qquad E^4=-(3e+3k+c).$$
    Since $Y$ is a cubic,  $(H')^4=3.$
    Using Equation~\ref{eqn:H in terms of H',E'} to substitute $H=aH'-bE'$, we see $H(H')^3=3a$. On the other hand, we can expand, using $H'=aH-bE$ to obtain
    $$H(H')^3=a^3H^4-3a^2bH^3E+3ab^2H^2E^2-b^3HE^3,$$
    and substituting the above intersection numbers gives 
    $$3a= 3a^3-3ab^2e+b^3(3e+k).$$ Using $a=3b-1$, we can write $k$ in terms of $b$ and $e$ as
    \begin{equation}\label{eqn: k}
    k=\left(6-\frac3b\right) e-81+\frac{81}b-\frac{18}{b^2}.
    \end{equation}
    Using $(H')^4=3$, expanding, and substituting $a=3b-1$ and the formula for $k$ above gives the following formula for $c$:
    \begin{equation}\label{eqn: c}
    c=\left(33-\frac{27}{b}+\frac{6}{b^2}\right)e-486+\frac{729}{b}-\frac{324}{b^2}+\frac{36}{b^3}.
    \end{equation}

    Note that $H^2(H'^2)=3a^2-b^2e,$ again using $H'=aH-bE$ and the intersection numbers on $Z$. On the other hand, using $H=aH'-bE'$, we also obtain 
    $H^2(H'^2)=3a^2-b^2\deg T.$ 
    Since $b>0,$ we can conclude that $e= \deg T,$ meaning $S$ and $T$ have the same degrees.

    We can compute
    $c_2(Z)=p^*c_2(X)+p^*[S]-p^*c_1(X)E$ using \cite[Example 15.4.3]{fulton} which gives us $c_2(Z)=6H^2+p^*[S]-3HE$. 
    Using blow-up formulas,
    \begin{align*}
        c_2(Z)H^2&= 18 +e, \\
        c_2(Z)HE&= 3e,\\
        c_2(Z)E^2&=-6e-\deg c_2(N_{S/X})+3(3e+k)\\
        &=3e+3k-\deg c_2(N_{S/X}).
    \end{align*}
    Since $\deg T =e,$ we also obtain $c_2(Z)(H')^2= 18 +e,$ which we can expand with $H'=aH-bE$ to get
    $$(a^2-1)(18+e)-6abe+b^2(3e+3k- \deg c_2(N_{S/X}))=0.$$
    Substituting $a=3b-1$ yields $$b^2(162-6e+3k-\deg c_2(N_{S/X}))-108 b=0,$$ so $$\deg c_2(N_{S/X})= 162-6e+3k-\frac{108}{b}.$$
    On the other hand, $\deg c_2(N_{S/X})= 6e+3k+K_S^2-c$, so
    \begin{equation}\label{eqn: K2} 
    K_S^2=c-12e+162-\frac{108}{b}.
    \end{equation} 
    Applying Noether's formula, $K_S^2+c=12\chi$, yields
    \begin{equation*}
        6\chi=c-6e+81-\frac{54}{b}.
    \end{equation*}
    Substituting Equation \ref{eqn: c} for $c$ and factoring gives:
    \begin{equation}\label{eqn: chi}
        2b^3\chi=(3b-2)(b(3b-1)e-45b(b-1)-6).
    \end{equation}
    
    We now analyse the numerics imposed by the relations we have found so far. Reducing Equation~\ref{eqn: chi} modulo $b$ gives $0=12 \mod b$, so $b| 12$. Similarly, multiplying Equation \ref{eqn: k} by $b^2$ and reducing modulo $b$ gives $0\equiv -18 \mod b$, so $b\mid 18$. We are left with four cases: $b\in \{1,2,3,6\}$.

    We computed earlier that $H^2(H')^2= 3a^2-b^2e$ and $H^3H'=H(H')^3=3a$. The Khovanskii-Teissier inequality \cite[Example 1.6.4]{Lazarsfeld1} gives
    $$(H^2(H')^2)^2\geq (H^3H')(H(H')^3)=9a^2.$$ Since $H$ and $H'$ are nef, $H^2(H')^2\geq 0$, so $3a^2-b^2e\geq 3a.$ Thus we can bound the degree of $S$:
    \begin{equation}\label{eqn:bound}
        e\leq \frac{3a(a-1)}{b^2}.
    \end{equation}
The following table records the resulting bound on $e$ for each possible value of $b$:
\begin{table}[h!]
    \centering
    \begin{tabular}{c|cccc}
        $b$ & $1$ & $2$ & $3$ & $6$ \\
        \hline
        upper bound on $e$ & $6$ & $15$ & $18$ & $22$
    \end{tabular}
    \label{tab: bounds on the degree}
\end{table}

Next, we look at the numerics of $S$. We write $H$ for the restriction of the hyperplane class to $S$.
We first claim that $k^2- e K_S^2\geq 0.$ 
Indeed, letting $D=eK_S-k H\in \Pic(S),$ we see that $H\cdot D=0.$ Thus by the Hodge index theorem, $D^2\leq 0,$ and expanding gives the claim.
Equations \ref{eqn: k}, \ref{eqn: c}, and \ref{eqn: K2} allow us to compute $k^2-eK_S^2$ in terms of $e$ and $b$. For each possible value of $b$, we record these computations, along with the sectional genus 
$g=1+(e+k)/2\geq0$:
\begin{center}
\small
\renewcommand{\arraystretch}{1.4}
\setlength{\tabcolsep}{5pt}
\begin{tabular}{c|c|c|c|c|c}
$b$ & $k$ & $c$ & $K_S^2$ & $k^2-eK_S^2$ & $g$ \\
\hline
$1$ & $3e-18$ & $12e-45$ & $9$ & $9(e-9)(e-4)$ & $2e-8$\\
$2$ & $\frac92e-45$ & $21e-198$ & $9e-90$ & $\frac{45}{4}(e-18)(e-10)$ & $\frac14(11e-86)$\\
$3$ & $5e-56$ & $\frac{74}3e-\frac{833}3$ & $\frac{38}3e-\frac{455}3$ & $\frac13(e-21)(37e-448)$ & $3e-27$\\
$6$ & $\frac{11}2e-68$ & $\frac{86}3e-\frac{1120}3$ & $\frac{50}{3}e-\frac{688}3$ & $\frac1{12}(e-24)(163e-2312)$ & $\frac{13}4e-33$
\end{tabular}
\end{center}

We now analyse constraints. Equation~\ref{eqn:bound} bounds the positive integer $e$ in terms of $b$, and we also know that $g$ and $k^2-eK_S^2$ are non-negative integers. This leaves us with seven possibilities, namely
\[
(b,e)\in\{(1,4),(2,10),(3,9),(3,10),(3,11),(3,12),(6,12)\}.
\]
A further constraint comes from the fact that $\chi$ is an integer; indeed, when $b=3$, Equation~\ref{eqn: chi} yields
\[
\chi=\frac{14(2e-23)}9,
\]
which is not an integer when $9\le e\le12$; in particular, we rule out all four cases with $b=3$. Similarly, substituting $b=6$ and $e=12$ in Equation~\ref{eqn: chi} yields $\chi=-44/9$, ruling out the last case.

Only two possibilities remain: $(b,e)=(1,4)$ or $(2,10)$, with the invariants recorded below:
\[
 \begin{array}{c|c|c|c|c|c}
 b&e&k&c&\chi&K_S^2\\\hline
 1&4&-6&3&1&9\\
 2&10&0&12&1&0
 \end{array}
\]
In particular, as $a=3b-1$, we deduce that $H'=2H-E$ (if $b=1$) or $H'=5H-2E$ (if $b=2$), as desired.
We finish by identifying the surface $S$ in each case, using standard results in the classification of surfaces; see, for example, \cite{beauville} for details. Note that a symmetric argument shows that $T$ has the same numerics and type as $S$.

In the first case, $k=H\cdot K_S<0$, so $S$ is birationally ruled. Because $\chi=1$, $S$ is also rational. Since $c:= c_2(S) =3,$ we see that $S\cong \bP^2$, embedded into $\bP^5$ with degree 4. This is the Veronese surface, so $X\in\calC_{20}$. The birational link is known from \cite{FL23}.

In the second case, $k=H\cdot K_S=0$, and $K_S^2=0$; the Hodge index theorem implies that $K_S\equiv 0$. Since $\chi=1,$ the classification of surfaces tells us $S$  is an Enriques surface, embedded into $\bP^5$ with degree 10, so $X\in\calC_{44}$. This is the birational link identified in Theorem~\ref{thm: cubic theorem}.
\end{proof}

\section{An Enriques Cremona transform of $\bP^5$}\label{sec: cremona}
Let $S\subset \bP^5$ be a very general unnodal Enriques surface embedded in $\bP^5$ via a Fano polarisation.
Let $V:=H^0(\bP^5, \calI^2_S(5))$, which has dimension 6 by Proposition~\ref{prop: dim of quintics}.
Let $f:\bP^5\dashrightarrow \bP^5:=\bP(V^\vee)$ be the rational map defined by $V$.

\begin{theorem}\label{thm: Cremona}
   With the notation above, the rational map $f$ admits a factorisation:
\[\begin{tikzcd}
    \bP^-\arrow[r,dashrightarrow, "\varphi"
    ]\arrow[d, "\pi_S"]& \bP^+\arrow[d, "\pi_T"]\\
    \bP^5 \arrow[r,dashrightarrow, "f"]& \bP^5
\end{tikzcd}\]
    where:
    \begin{itemize}
        \item $\pi_S:\bP^-:=\Bl_S\bP^5\rightarrow \bP^5$ is the blow up of the smooth Enriques surface $S$.
        \item $\pi_T:\bP^+:=\Bl_T\bP^5\rightarrow \bP^5$ is the blow up of the smooth Enriques surface $T$ constructed in Lemmas \ref{lem: q is a blowup}, and \ref{lem: T is Enriques}.
        \item $\varphi:\bP^-\dashrightarrow \bP^+$ is the simultaneous standard flop of the 20 disjoint planes $\widetilde{\Pi}_i$.
    \end{itemize}
    Moreover, the rational inverse to $f$ is given by the linear system of quintics containing $T$ with multiplicity $2$.
\end{theorem}

The first step to proving Theorem~\ref{thm: Cremona} is an analysis of the normal bundles to the planes $\widetilde{\Pi}_i\subset\bP^-$.

\begin{lemma}\label{lem: global generation}
     The vector bundle $N_{\widetilde{\Pi}_i/\bP^-}^\vee(-1)$ is globally generated.
\end{lemma}
\begin{proof}
    Let $\mu:\widetilde{\bP}\rightarrow \bP^{-}$ be the blow up of the disjoint union of the $20$ planes $\widetilde{\Pi}_i$, with exceptional divisors $G_i$. 
    Each exceptional divisor $G_i$ is isomorphic to $\bP(N_{\widetilde{\Pi}_i/\bP^-})$, with $\mu$ restricting on $G_i$ to the natural projection $\mathrm{pr}_i: G_i\to \widetilde{\Pi}_i$. Recall that $\calO_{G_i}(1)=\calO_{\widetilde{\bP}}(-G_i)|_{G_i}$ and that $(\mathrm{pr}_i)_*\calO_{G_i}(1)=N_{\widetilde{\Pi}_i/\bP^-}^\vee$.

    Letting $D=5H-2E$ and $L=\mu^*D -\sum G_i$, we see that $L$ is globally generated line bundle on $\widetilde{\bP}$. Indeed this follows since the scheme-theoretic base locus of $|D|$ is the union of the planes $\widetilde{\Pi}_i$ by Proposition \ref{prop: base locus}. Since $\mu^*D|_{G_i}=\mathrm{pr}_i^*(D|_{\widetilde{\Pi}_i})=\mathrm{pr}_i^*(-l_i)$, restricting $L$ to $G_i$ gives
    $$L|_{G_i}=\mathrm{pr}_i^*\calO_{\widetilde{\Pi}_i}(-1)\otimes \calO_{G_i}(1).$$
    This line bundle is also globally generated. 
    
    Now, suppose toward contradiction that $N^\vee_{\widetilde{\Pi}_i/\bP^-}(-1)$ is not globally generated. Then there is some point $p\in\widetilde{\Pi}_i$ for which the image of the evaluation map
    \[
    \mathrm{ev}_p:H^0(\widetilde{\Pi}_i, N_{\widetilde{\Pi}_i/\bP^-}^\vee(-1))\to (N_{\widetilde{\Pi}_i/\bP^-}^\vee)_p\otimes\calO_{\widetilde{\Pi}_i}(-1)_p
    \]
    fails to surject. It follows that there is some $v\in(N_{\widetilde{\Pi}_i/\bP^-})_p$ such that $\mathrm{ev}_p(s)(v)=0$ for all $s\in H^0(\widetilde{\Pi}_i, N_{\widetilde{\Pi}_i/\bP^-}^\vee(-1))$, where we regard $\mathrm{ev}_p(s)$ as an element of $\Hom((N_{\widetilde{\Pi}_i/\bP^-})_p,\calO_{\widetilde{\Pi}_i}(-1)_p)$. On the other hand, the projection formula gives a natural isomorphism 
    \[
    H^0(G_i,L|_{G_i})\cong H^0(\widetilde{\Pi}_i,N_{\widetilde{\Pi}_i/\bP^-}^\vee(-1)),
    \]
    and under this identification, regarding $[v]$ as a point in $G_i\cong\bP(N_{\widetilde{\Pi}_i/\bP^-})$, we see that the image of
    \[
    \mathrm{ev}_{[v]}:H^0(G_i,L|_{G_i})\to L_{[v]}
    \]
    is trivial, so $[v]$ is a basepoint of $L|_{G_i}$, a contradiction.
\end{proof}

\begin{lemma}\label{lem: norm bundle plane}
    Let $\widetilde{\Pi}_i\subset \bP^-$ be the strict transform of $\Pi_i$, one of the planes swept out by trisecant lines to $S$. Then $N_{\widetilde{\Pi_i}/ \bP^-}\cong \calO_{\bP^2}(-1)^{\oplus 3}$.
\end{lemma}
\begin{proof}
    First, note that $K_{\bP^-}=-6H+2E$ since $\pi_S:\bP^-\to \bP^5$ is the blow up of a smooth codimension $3$ centre. Further, $H|_{\widetilde{\Pi}_i}= l_i$ where $l_i$ is the class of a line $l_i\subset \widetilde{\Pi}_i$, and $E|_{\widetilde{\Pi}_i}=3l_i$, since $\Pi_i\cap S=F_i$ is a cubic plane curve. It follows that $K_{P^-}|_{\widetilde{\Pi}_i}=0$. On the other hand, adjunction gives $$-3l_i=K_{\widetilde{\Pi}_i}=(K_{\bP^-}+\det N_{\widetilde{\Pi}_i/\bP^-})|_{\widetilde{\Pi}_i}. $$ Thus we have $\det N_{\widetilde{\Pi}_i/\bP^-}=\calO_{\bP^2}(-3)$. Since $N^\vee_{\widetilde{\Pi}_i/\bP^-}(-1)$ is globally generated by Lemma~\ref{lem: global generation} and $\det(N^\vee_{\widetilde{\Pi}_i/\bP^-}(-1))=\calO$, we conclude that $N^\vee_{\widetilde{\Pi}_i/\bP^-}(-1)=\calO^{\oplus 3}$, and the result follows.
\end{proof}

We are now prepared to prove the main result of this section.

\begin{lemma}\label{lem: factor along flops}
    With notation as above, there is a factorisation
    \[\begin{tikzcd}
    \bP^-\arrow[r,dashrightarrow, "\varphi"
    ]\arrow[d, "\pi_S"]& \bP^+\arrow[d, "\sigma"]\\
    \bP^5 \arrow[r,dashrightarrow, "f"]& \bP^5
    \end{tikzcd}\]
    where $\pi_S$ is the blowup along $S$, $\varphi$ is a standard flop of the twenty disjoint planes $\widetilde{\Pi}_i$, and under the isomorphism $\Pic(\bP^-)\cong\Pic(\bP^+)$ given by the flop, $\sigma$ is induced by $|5H-2E|$.
\end{lemma}
\begin{proof}
    First, note that the linear system $|D|=|5H-2E|$ defines a rational map $\bP^-\dashrightarrow \bP^5$ whose image coincides with that of $f$. Further, by Proposition \ref{prop: base locus}, the scheme-theoretic base locus of $|D|$ is the disjoint union of twenty planes $\widetilde{\Pi}_i$.

Note that $$D|_{\widetilde{\Pi}_i}= 5H|_{\widetilde{\Pi}_i}-2E|_{\widetilde{\Pi}_i} =5l_i-6l_i=-l_i, $$ where $l_i$ is the class of a line.

    Consider $M=3H-E;$ since $\calI_S(3)$ is globally generated, the linear system $|M|$ is globally generated on $\bP^-.$ Further, $M|_{\tilde{\Pi}_i}=3l-3l=0$ for each $i$. We claim $|M|$ defines a morphism contracting exactly the union of the planes $\widetilde{\Pi}_i$. Indeed, suppose $C$ is a contracted curve not contained in their union. By Proposition \ref{prop: base locus}, some member of $|D|$ does not contain $C$, and $D\cdot C\geq 0.$ Note $2M=H+D$; since $H$ is nef, this implies that $H\cdot C=D\cdot C=0$. The first implies $C$ is contained in a fibre $F$ of $\pi_S:\bP^-\rightarrow \bP^5,$ but $D|_{F}=(-2E)|_{F}=\calO_F(2),$ a contradiction.
    Thus, taking the Stein factorisation of the morphism given by $|M|$, we have a small contraction $\bP^-\rightarrow \overline{\bP}$ contracting exactly the union of the twenty planes. Note is is a $K_{\bP^-}$-trivial contraction, since $-K_{\bP^-}=2M.$ 

    Thus we can apply \cite[Proposition 1.3]{flops} to conclude the ordinary flop of each of the $\widetilde{\Pi}_i\cong \bP^2$ exists. Its common resolution is the blow up $\mu:\widetilde{\bP}\rightarrow \bP^{-}$ of the disjoint union of the $20$ planes $\widetilde{\Pi}_i$, with exceptional divisors $G_i$. Using Lemma~\ref{lem: norm bundle plane}, we have  $G_i\cong \bP(N_{\widetilde{\Pi}_i/\bP^-})\cong \bP^2\times \bP^2$, with $\mu$ restricting to the projection to the first factor.
    The exceptional divisors $G_i$ are blown down in the other direction by \cite[Section 1]{flops}, defining the flop $\varphi:\bP^-\dashrightarrow\bP^+$:
    $$\begin{tikzcd}
    &\widetilde{\bP}\arrow[dr, "\mu_+"]\arrow[dl, "\mu" above]&\\
        \bP^-\arrow[rr, dashrightarrow, "\varphi"]&&\bP^+
    \end{tikzcd}$$

    Since the base locus of $|D|$ is the union of the planes $\widetilde{\Pi}_i$, the line bundle $L=\mu^*D -\sum G_i$ is globally generated and defines a morphism $\widetilde{\bP}\rightarrow \bP^5$. 
    We claim this factors through the contraction $\mu_+$. 
    First, observe that $\calO_{\widetilde{\bP}}(-G_i)|_{G_i}\simeq \calO_{G_i}(1)\simeq\calO_{\bP^2\times \bP^2}(1,1)$ and $\mu^*D|_{G_i}=\mu^*(D|_{\widetilde{\Pi}_i})=\calO_{\bP^2\times \bP^2}(-1,0)$, so 
    \[
    L|_{G_i}\simeq\mu^*D|_{G_i}\otimes \calO_{\widetilde{\bP}}(-G_i)|_G\simeq\calO_{\bP^2\times \bP^2}(0,1).
    \]
    Thus $L$ is trivial on the fibres of $\widetilde{\bP}\rightarrow \bP^+$, and so the morphism $|L|:\widetilde{\bP}\rightarrow \bP^5$ factors through $\mu_+$:
    $$\begin{tikzcd}
        \widetilde{\bP}\ar[r, "\mu_+"]& \bP^+\ar[r, "\sigma"]& \bP^5.
    \end{tikzcd}$$

    The globally generated line bundle $L$ descends to a globally generated line bundle $H'$ on $\bP^+$ with $L=\mu_+^{*}H'$.
    Since the flop $\varphi:\bP^-\dashrightarrow \bP^+$ is an isomorphism in codimension one, it follows that $\Pic(\bP^-)\cong \Pic(\bP^+)$, and under this identification, $H'=5H-2E$, and the morphism $\sigma$ is defined by $|H'|$; that is, $\sigma^*\calO_{\bP^5}(1)=H'$. 
    \end{proof}

    We now prove that $\sigma$ coincides with the blowup of the Enriques surface $T\subset \bP^5$ from Theorem~\ref{thm: cubic theorem}.

    \begin{lemma}\label{lem: sigma is a blowup}
    The morphism $\sigma:\bP^+\to\bP^5$ constructed in Lemma~\ref{lem: factor along flops} is the blowup of $\bP^5$ along an Enriques surface $T\subset \bP^5$.
    \end{lemma}
    \begin{proof}
    We retain the notation from the proof of Lemma~\ref{lem: factor along flops}. The first step is to show that $\sigma$ is birational by computing $H'^5=1$. 
    Note that $L^5=(\mu_+^*H')^5=H'^5$, so it suffices to compute $L^5$ on $\widetilde{\bP}$. 
    Toward this end, we compute $D^5$ on $\bP^-$. As in the proof of Lemma~\ref{lem: int no. Z}, we compute intersection numbers on $\bP^-$, using Lemma \ref{lem: chern classes}:
    $$H^5=1, \qquad H^2E^3=10,\qquad HE^4=60,\qquad E^5=222, $$
    and $H^4E=H^3E^2=0$. It follows that $D^5=21$.

    Then $L^5=(\mu^*D- \sum G_i)^5$. 
    Note $\mu$ is the blow up of a codimension three smooth centre. 
    It follows that $\mu_*G_i= \mu_*G_i^2=0$, whereas the blow-up formulas in \cite[Section 13.6]{3264} yield
    \begin{align*}
      \mu_*(G_i^3)&=[\widetilde{\Pi}_i],  \\
      \mu_*(G_i^4)&=c_1(N_{\widetilde{\Pi}/ \bP^-}),\text{ and }\\
      \mu_*(G_i^5)&=c_1(N_{\widetilde{\Pi}/ \bP^-})^2-c_2(N_{\widetilde{\Pi}/ \bP^-}).
    \end{align*}
    By Lemma \ref{lem: norm bundle plane}, $c_1(N_{\widetilde{\Pi}/ \bP^-})=-3l_i$ and $c_2(N_{\widetilde{\Pi}/ \bP^-})=3l_i^2$. 
    The divisors $G_i$ are disjoint, so all cross terms are trivial. We now compute
    \begin{align*}
        (\mu^*D)^2G_i^3 & = D^2\mu_*(G_i^3)=D^2|_{\widetilde{\Pi}_i}= l_i^2=1,\\
        (\mu^*D)G_i^4&= D\mu_*(G_i^4)=D|_{\widetilde{\Pi}_i}c_1(N_{\widetilde{\Pi}/ \bP^-})= (-l_i)(-3l_i)=3, \text{ and }\\
        (G_i^5)&=(-3l_i)^2-3= 6.
    \end{align*}
    It follows that $H'^5=L^5=1$, as desired. Since $H'$ is globally generated and has positive top self-intersection, $\sigma$ is dominant and of degree $1$, hence birational.
    
    We now analyse the exceptional locus of $\sigma$. Let $X\subset \bP^5$ be a very general cubic fourfold containg $S$, with strict transform $Z\subset \bP^-$. 
    The flop $\varphi$ restricts to an isomorphism on $Z$, since $Z$ is disjoint from the planes by Proposition \ref{prop: image is a cubic}. 
    We can thus regard $Z\subset \bP^+$ as a divisor. 
    It follows that $\sigma|_Z:Z\rightarrow Y\subset \bP^5$ coincides with the morphism $q$ constructed in Theorem \ref{thm: cubic theorem}, i.e. the blow up of some Enriques surface $T$ contained in a smooth cubic fourfold $Y$.

    Set $E_{\bP^+}':=12H - 5E$ considered as a divisor in $\bP^+$. We claim $E'_{\bP^+}$ is the unique exceptional divisor of $\sigma:\bP^+\rightarrow \bP^5$. 
    By a similar argument to Lemma \ref{lem: q is a blowup},  since both $\bP^+$ and $\bP^5$ are smooth and the relative Picard number is 1,  $\sigma$ is an elementary divisorial contraction and its exceptional locus is an irreducible divisor, say $F$.
    Since $Y\subset\bP^5$ is a cubic and $\sigma$ maps $Z$ birationally onto $Y$, the total transform of $Y$ is $Z+ mF$ with $m>0$. Thus
    $mF\sim3H'-(3H-E)=12H-5E$, since $Z\sim 3H-E$ and $\sigma^*Y\sim 3H'$ both considered in $\Pic(\bP^+)$. Since $12H-5E$ is primitive, it follows that $m=1$ and $F=E'$.

    By \cite[Lemma 3.39]{KolMori}, the effective divisor $E'$ is not $\sigma$-nef. Since the relative Picard number is 1, it follows that $-E'$ is $\sigma$-ample, and $Z\sim 3H'-E'$ is also $\sigma$-ample.
    
    Let $B:=\sigma(E')$; by restricting to $Z$ we see that $T\subset B$. Now, suppose $b\in B$, and let $F=\sigma^{-1}(b)\subset E'$. Since $Z$ is $\sigma$-ample, there is some point $z\in F\cap Z$. But then $z$ lies in the exceptional locus of $\sigma|_Z:Z\to Y$, and it follows from Theorem~\ref{thm: cubic theorem} that $b=\sigma(z)\in T$. Hence $B=T$.
    Finally, we apply \cite[Proposition~2.2]{AW98II} to conclude that $\sigma$ coincides with the blowup $\pi_T:\Bl_T\bP^5\to \bP^5$. 
\end{proof}

The only detail remaining to complete the proof of Theorem~\ref{thm: Cremona} is a description of the rational inverse to $f$.

\begin{lemma}
    The rational map $g:\bP^5\dashrightarrow\bP^5$ given by the complete linear system of quintics containing $T$ with multiplicity $2$ provides a rational inverse to $f$, in the sense that $g\circ f$ is an isomorphism.
\end{lemma}
\begin{proof}
    Since $T$ is also an Enriques surface with its Fano embedding, we can apply Lemmas~\ref{lem: factor along flops} and~\ref{lem: sigma is a blowup} to achieve the following diagram, where $\pi_{S'}:\bP^!\to\bP^5$ is the blowup along an Enriques surface $S'$ and $\varphi'$ is a standard flop of twenty disjoint planes:
    \[\begin{tikzcd}
    \bP^-\arrow[r,dashrightarrow, "\varphi"
    ]\arrow[d, "\pi_S"]& \bP^+\arrow[d, "\pi_T"] \arrow[r,dashrightarrow,"\varphi'"] & \bP^! \arrow[d,"\pi_{S'}"]\\
    \bP^5 \arrow[r,dashrightarrow, "f"]& \bP^5 \arrow[r,dashrightarrow,"g"] & \bP^5
    \end{tikzcd}\]
    Let $H''=\pi_{S'}^*\calO_{\bP^5}(1)$, and let $E''$ be the exceptional divisor of $\bP^!$. Under the identifications $\Pic(\bP^-)\cong\Pic(\bP^+)$ and $\Pic(\bP^+)\cong\Pic(\bP^!)$, we have $H'=5H-2E$, $E'=12H-5E$, and $H''=5H'-2E'$. It follows that $H''=H$, so $g\circ f$ is given by the complete linear system $|\calO_{\bP^5}(1)|$, and the result follows.
\end{proof}

    It is natural to wonder how the two Enriques surfaces $S$ and $T$ compare, although it is not clear to us what to expect. We do, however observe an interesting relationship in the Grothendieck ring of varieties: the scissor relations
    \[
    [\bP^-]=[\bP^5]-[S]+[E]=[\bP^5]+([\bP^2]-1)[S]
    \]
    and
    \[
    [\bP^+]=[\bP^5]-[T]+[E']=[\bP^5]+([\bP^2]-1)[T],
    \]
    together with the relation
    \[
    [\bP^-]=[\widetilde\bP]-20[\bP^2]^2+20[\bP^2]=[\bP^+]
    \]
    yields
    \[
    (\bL^2+\bL)([S]-[T])=0,
    \]
    where $\bL=[\bA^1]$. This implies that if $[S]$ and $[T]$ are $\bL$-equivalent, i.e. if their difference is annihilated by some power of $\bL$, then $\bL[S]=\bL[T]$. Similarly, if $X$ is a very general cubic fourfold containing $S$ and $Y$ is the cubic constructed in Theorem~\ref{theorem: main} containing $T$, then we also have 
    \[
    [X]-[Y]=\bL([T]-[S]).
    \]
    In particular, we arrive at the following:

\begin{corollary}\label{cor: L equivalence}
    Let $S$ and $T$ be Enriques surfaces related by an Enriques Cremona transformation, as in Theorem~\ref{thm: Cremona}, and let $X$ be a very general cubic fourfold containing $S$ and $Y$ be its image under the Enriques Cremona transform.
    The following are equivalent:
    \begin{itemize}
        \item $S$ and $T$ are $\bL$-equivalent,
        \item $X$ and $Y$ are $\bL$-equivalent,
        \item $\bL([S]-[T])=0$, and
        \item $[X]=[Y]$ in $K_0(\mathrm{Var}_\bC)$.
    \end{itemize}
\end{corollary}

\begin{remark}
For further discussion of $\bL$-equivalences between cubic fourfolds, see \cite{MeinsmaMoschetti}, where the authors conjecture that $\bL$-equivalent cubics are Fourier--Mukai partners. The authors moreover ask whether there are nontrivially $\bL$-equivalent cubic fourfolds, i.e. $\bL$-equivalent cubics $X$ and $Y$ such that $[X]\neq[Y]$ in $K_0(\mathrm{Var}_\bC)$. Corollary~\ref{cor: L equivalence} shows that no such example will be found by looking at Fourier--Mukai partnerships in $\calC_{44}$, at least for the very general member.
\end{remark}

\section{The Fano variety of lines}\label{sec: Fano}
Let $X$ be a smooth cubic fourfold. Then the variety of lines on $X$, denoted by $F(X)$, is a smooth hyperk\"ahler fourfold of K3$^{[2]}$ type \cite{BeauvilleDonagi}. The second cohomology $H^2(F(X),\bZ)$ is equipped with a quadratic form $q$, called the Beauville--Bogomolov--Fujiki (BBF) form. In particular, $F(X)$ inherits the Pl\"ucker polarisation $g$ from its embedding in $\Gr(2,6)$, which satisfies $q(g)=6$ and $\mathrm{div}(g)=2$ \cite[Proposition 6]{BeauvilleDonagi}.

In this section, we will prove the following:

\begin{theorem}\label{thm: bir Fano}
    Let $X$ be a very general cubic fourfold $X\in \calC_{44}$, and let $Y$ be the partner cubic constructed as in Theorem \ref{thm: cubic theorem}.
    Then $Y$ is the unique cubic fourfold not isomorphic to $X$ such that $F(X)\cong F(Y)$.
\end{theorem}

If $X$ and $Y$ are two smooth cubic fourfolds, the property of having birational Fano varieties of lines is closely related to Fourier--Mukai partnership of the cubic fourfolds. In Section \ref{subsec: BF implies FM v general}, we prove that if $X$ is very general in $\calC_d$, and $F(X)$ is birational to $F(Y)$, then $X$ and $Y$ are Fourier--Mukai partners. In \cite{BFM}, the authors conjecture that a birational equivalence between $F(X)$ and $F(Y)$ implies that $X$ and $Y$ are birational. In the case of irrational cubic fourfolds, one may expect these two conditions to be equivalent, and thus one expects Theorem \ref{thm: bir Fano}, proved in Section \ref{subsec: bir fanos}.

\subsection{Relationship to Fourier--Mukai partner}\label{subsec: BF implies FM v general}

Recall $\tilde{\Lambda}:=E_8(-1)^{\oplus 2}\oplus U^{\oplus 4}$ is the extended K3 lattice. We follow the discussion of \cite[Proof of Proposition 4.1]{HuyK3cat} (see also \cite{markman}).

Recall that there exists a distinguished primitive embedding $H^2(F(X),\bZ)\hookrightarrow\tilde{\Lambda}$ whose image has an orthogonal complement spanned by a class $v\in \tilde{\Lambda}$ with $v^2=2$ (see \cite{markman} and also \cite[Proof of Proposition 4.1]{HuyK3cat}. One can equip $\tilde{\Lambda}$ with an induced Hodge structure from $H^2(F(X),\bZ)$ such that $v$ is algebraic. 

By \cite[Corollary 8]{Add16}, this embedding coincides for the distinguished embedding 
$$H^2(F(X),\bZ)\simeq \lambda_1^\perp\hookrightarrow \widetilde{H}(\Ku(X),\bZ)\cong \widetilde{\Lambda}. $$

If $F(X)$ and $F(Y)$ are birational, then there exists a Hodge isometry between $H^2(F(X),\bZ)$ and $H^2(F(Y),\bZ)$ extending to a Hodge isometry $\widetilde{\Lambda}_X\simeq \widetilde{\Lambda}_Y$, where the subscripts denote induced Hodge structures. We therefore obtain a Hodge isometry $\widetilde{H}(\Ku(X),\bZ)\simeq \widetilde{H}(\Ku(Y),\bZ)$.

We obtain:

\begin{corollary}\label{cor:BF imply FM}
    Let $X$ be very general in $\calC_d$, and suppose that $F(X)$ is birational to $F(Y)$ for a smooth cubic fourfold $Y$. Then $Y$ is a Fourier--Mukai partner of $X$.
\end{corollary}
\begin{proof}
    As we just argued, if $F(X)$ is birational to $F(Y)$, then we obtain a Hodge isometry $\tilde{H}(\Ku(X),\bZ)\simeq \tilde{H}(\Ku(Y),\bZ)$. Applying Theorem \ref{thm: huy17} completes the proof.
\end{proof}

\begin{remark}
    The converse does not hold in general. Indeed, in \cite[Section 6]{BFM}, the authors prove that a very general cubic fourfold $X$ in $\calC_{546}$ has a nontrivial Fourier--Mukai partner $Y$, but $F(X)$ and $F(Y)$ are not birational. The key here is that the Hodge isometry $\tilde{H}(\Ku(X),\bZ)\to \tilde{H}(\Ku(Y),\bZ)$ does not restrict to a Hodge isometry between $H^2(F(X),\bZ)$ and $H^2(F(Y),\bZ)$.
\end{remark}

\subsection{Birational geometry of $F(X)$}\label{subsec: bir fanos}
Let $X$ be very general in $\calC_{44}$ and $F(X)$ the Fano variety of lines on $X$. Let $\alpha:H^4(X,\bZ)\to H^2(F(X),\bZ)$ be the Abel--Jacobi map, which restricts to an anti-isometry on primitive cohomology.
Set $g=\alpha(\eta_X)$, $s:=\alpha([S])$, and $\lambda:=s-3$. We can write the Gram matrix for the BBF form on $\NS(F(X))=\langle g, \lambda\rangle$ as
$$J_{44}:=\left(\begin{matrix}
    6 & 2\\
    2 & -14
\end{matrix}\right). $$

Let $\overline{\mathrm{Pos}(F(X))}$ be the component of the cone $\{x\in \NS(F(X))\otimes \bR\mid q(x)\geq 0\}$ containing an ample class. 

\begin{proposition}\label{prop: one min model}
    For a very general $X\in \calC_{44}$, the Fano variety of lines $F(X)$ has only one minimal model.
\end{proposition}
\begin{proof}
    We claim that $\overline{\Mov(F(X))}=\Amp(F(X))=\overline{\mathrm{Pos}(F(X))}$, implying the claim via the Kawamata Morrison cone theorem for hyperk\"ahler fourfolds \cite{AV1}.  Indeed, $\Mov(F(X))$ is a connected component of $$\overline{\mathrm{Pos}(F(X))}\setminus \bigcup_{\substack{\rho\in \NS(F(X)) \\  \rho^2=-2}}\rho^\perp, $$ and the nef cone is a connected component of $$\overline{\mathrm{Pos}(F(X))}\setminus \bigcup_{\substack{\rho\in \NS(F(X))\\ \rho^2\in\{-2,-10\}\\ \mathrm{div}(\rho)=2}}\rho^\perp;$$ see \cite[Theorem 3.16]{debarre2020hyperkahler} for instance in the case of K3$^{[2]}$ type. It is an elementary check to see that the lattice $J_{44}$ does not represent $-2$ or $-10$.
\end{proof}

Denote the discriminant group of $\NS(F(X))$ by $D:=\NS(F(X))^\vee/\NS(F(X))$. Note that 
$$D=\left\langle \frac{\lambda}{2}\right\rangle\oplus \left\langle \frac{g-3\lambda}{44}\right\rangle\cong \bZ/2\bZ\oplus \bZ/44\bZ. $$
We have the following: 
\begin{lemma}\label{lem: action on subgroup}
    An isometry $\phi\in \mathrm{O}(\NS(F(X)))$ is induced by a birational automorphism of $F(X)$ if and only if $\varphi$ preserves the positive cone and acts by $\pm \id$ on the subgroup
    $H=\langle \frac{g-3\lambda}{44}\rangle\cong \bZ/44\bZ\leq D(\NS(F))$.
\end{lemma}
\begin{proof}
    The proof follows identically to that of \cite[Lemma 4.2]{BFM}.
\end{proof}

We now prove the main result of this section.

\begin{proof}[Proof of \ref{thm: bir Fano}]
    Consider the class $g'=5g+4\lambda$ in $\NS(F(X))$. This class is primitive and satisfies $q(g')=6$ and $\mathrm{div}(g')=2$. By \cite[Proposition 2.3]{BFM}, there exists a smooth cubic fourfold $X'$ and a birational equivalence $\pi:F(X')\dashrightarrow F(X)$ such that $\pi^*(g')$ is the Pl\"ucker class of $F(X')$. By Proposition \ref{prop: one min model}, $\pi$ is an isomorphism. To show $X\not\cong X'$, it suffices to show that
    $$\phi:=\left(\begin{matrix}
        5 & -6\\
        4 &-5
    \end{matrix}\right), $$ which is the unique isometry of $\NS(F(X))$ sending $g$ to $g'$, is not induced by an automorphism of $F(X)$. Indeed, $X\cong X'$ if and only if there exists a polarised isomorphism $F(X)\cong F(X')$ \cite{charles12}.

    By Lemma \ref{lem: action on subgroup}, it suffices to compute the action of $\phi$ on $\langle \frac{g-3\lambda}{44}\rangle\leq D$. Let $\beta:= \frac{g-3\lambda}{44};$ one computes that $\phi(\beta)=23[\beta]\in D$. Since $23\not\equiv \pm 1 \mod 44$, the two ample classes $g$ and $g' \in \NS(F(X))$ are not in the same orbit under the action of $\Aut(F(X))$, and so $X\not\cong X'$.

    Now $X'$ is a  nontrivial Fourier--Mukai partner of $X$ by Corollary \ref{cor:BF imply FM}. Since $X'$ has a unique Fourier--Mukai partner by \cite[Proposition 2.6]{FL23}, we conclude that $X'\cong Y$, where $Y$ is the cubic fourfold constructed in Section~\ref{subsec: construction}.
\end{proof}

\appendix
\section{An explicit example.}\label{appendix}

The accompanying ancillary files contain \verb|MAGMA| code \cite{MAGMA} to verify several computations used in this article. We briefly explain the rationale below, giving an account of how the code produces examples over $\bF_{101}$ and then arguing that these examples lift to characteristic zero.

\subsection{Verification of Proposition \ref{prop: dim of quintics}}\label{app: quintics}
We compute a single example over $\bF_{101}$ of a Fano model $S\subset \bP^5$. We construct an Enriques surface as in \cite[Example: H126E7]{MagmaHandbook}. The built in function constructs a surface of degree 9 and genus 6 in $\bP^4$ that is isomorphic to an Enriques surface blown up at a single point. 
We obtain $S$ by recovering the minimal model using the function \verb|MinimalModelKodairaDimensionZero()|.
We verify the surface $S$ is smooth, and contained in a $\bP^5.$
Finally, we verify $h^0(\bP^5,\calI^2_S(5))=6$ in this example.

\subsection{Verification of Proposition \ref{prop: base locus}}\label{app: base locus}

To compute the scheme theoretic base locus, we let $L$ denote the linear system of quintics in $\bP^5$ double along $S$, and compute \verb|BaseScheme(L)|. We find this has dimension 2, and degree 60. Next, we remove the component supported on $S$, and compute that the residual base locus has degree 20. By Proposition \ref{prop: base locus}, this verifies the union $\sqcup \widetilde{\Pi}_i$ is the set theoretic base locus of the linear system $D$ on $\Bl_S\bP^5.$

To conclude, we show that the residual base locus, denoted \verb|Bextra| in the computation, coincides scheme-theoretically with the union of these twenty planes. The function \verb|ReducedSubscheme(Bextra)| would help with this, but was too computationally time consuming on our machines. 

Instead, we sliced \verb|Bextra| with a random $\bP^3$ in $\bP^5$ to obtain 20 points. We reconstructed the tangent planes to \verb|Bextra| at these points and verified that the union of these planes equals \verb|Bextra| scheme theoretically.

\subsection{Verification of Proposition \ref{prop: image is a cubic}}\label{app: smoothness}

We take a random smooth cubic fourfold $X$ containing the Enriques surface $S$ that we constructed. The quintics double along $S$ were computed earlier as \verb|Q:= Sections(L)|. We use \verb|Q| to define the map \verb|phi0:=map<X->P5p\mid Q>| and take the image \verb|Y:=Image(phi0,X)|. We verify this is a smooth cubic fourfold.

\subsection{Lifting the computations}\label{app: lift}
We now explain how to lift the examples constructed above from $\bF_{101}$ to $\bC$. 
Let $(S,A)$ denote the Enriques surface over $\bF_{101}$ with its Fano polarisation $A$.
By \cite{Lang}, $S$ admits a lift to characteristic zero, and by \cite[Theorem 3.3.1]{Sern}, there is no obstruction to lifting $A$ along this deformation since $H^2(S,\calO_S)=0$.  As very ampleness is an open condition, we obtain a family of Fano-polarised Enriques surfaces
\[
(\mathcal S,\mathcal A)\longrightarrow \Spec(R)
\]
over a finite extension of $\bZ_{101}$, whose special fibre is $(S,A)$ and whose generic fibre $(\mathcal{S}_\eta,\mathcal{A}_\eta)$ has characteristic zero. The line bundle $\mathcal A$ gives a relative embedding $\mathcal S\hookrightarrow \bP^5_R.$

We claim that the properties verified above hold also for the generic fibre.
Indeed, the computation in Section \ref{app: quintics} gives $H^1(\bP^5,\calI_S^2(5))=0$, so also $H^1(\bP^5,\calI_{\mathcal S_\eta}^2(5))=0$ by upper semicontinuity.

Note also that each plane $\Pi_i$ lifts to a plane $\mathcal P_i\subset\bP_R^5$. 
To see this, we recall that $F_i=\Pi_i\cap S$ is a smooth plane cubic and look at the long exact sequence in cohomology associated to the sequence
\[
0\to\calO_S\to\calO_S(F_i)\to N_{F_i/S}\to0.
\]
We find $H^1(S,\calO_S(F_i))=0$, so by \cite[Theorem 3.3.4]{Sern}, the unique section of $\calO_S(F_i)$ lifts to a plane cubic $\calF_i\subset\bP_R^5$. The plane $\calP_i$ spanned by a $\calF_i$ lifts $\Pi_i$.

We next consider the base locus $\mathcal B$ of the system of quintics double along $\mathcal S$.
Set 
\[
\mathcal B'=2\mathcal S\cup\bigcup_{i=\pm1,\dots\pm10}\calP_i,
\]
The proof of Proposition \ref{prop: base locus} gives the inclusion $\mathcal B'\subseteq\mathcal B.$
The computation in Section~\ref{app: base locus} verifies that this inclusion is an equality scheme-theoretically on the special fibre.
Consequently the coherent sheaf $\calI_{\mathcal B'}/\calI_{\mathcal B}$ has zero restriction to the special fibre. Its support is closed in the proper scheme $\bP^5_R$, so its image in $\Spec R$ is closed. It follows that $(\calI_{\mathcal B'}/\calI_{\mathcal B})_\eta=0$ as well. Thus
\[
\mathcal B_\eta
=
2\mathcal S_\eta
\cup
\bigcup_{i=\pm1,\dots,\pm10}\mathcal P_{i,\eta}
\]
scheme-theoretically.
This verifies the constructions of Section \ref{app: base locus} lift to the generic fibre.

For Section \ref{app: smoothness}, we note that the cubic $X$ containing $S$ lifts to a cubic $\mathcal{X}\subset\bP_R^5$ containing $\mathcal S$, as does the morphism $\Bl_{\mathcal{S}}\mathcal{X}\to\bP^5_R$ induced by the complete linear system of quintics double along $\mathcal{S}$. (Here, we use that 
\[
\mathcal{B}\cap\mathcal{X}=\bigcup_{i=\pm1,\dots,\pm10}\mathcal{F}_i\subset\mathcal{S}
\]
as in the proof of Proposition~\ref{prop: image is a cubic}). The image is a cubic fourfold $\mathcal{Y}$, and our computation verifies that the special fibre $Y$ is smooth. Since smoothness is an open condition, so is $\mathcal{Y}_\eta$.

\bibliographystyle{alpha}
\bibliography{bibliography}
\end{document}